\documentclass[11pt]{article}
\usepackage[utf8]{inputenc}
\usepackage[T1]{fontenc}
\usepackage[english]{babel}
\usepackage{amssymb}
\usepackage{amsmath}
\usepackage{amsthm}
\usepackage{amsfonts}
\usepackage{esint}
\usepackage{dsfont}
\usepackage{wasysym}
\usepackage{bbm}
\usepackage[all]{xy}
\usepackage{hyperref}
\usepackage{caption}
\usepackage{enumitem}
\usepackage{graphicx}
\usepackage{tikz-cd}
\usepackage{tkz-tab}
\usetikzlibrary{decorations.markings,decorations.pathreplacing,positioning,cd}
\usepackage{graphicx}

\usepackage{stmaryrd}

\newtheorem{theorem}{Theorem}[section]
\newtheorem{corollary}[theorem]{Corollary}
\newtheorem{assumption}[theorem]{Assumption}
\newtheorem{lemma}[theorem]{Lemma}
\newtheorem{claim}[theorem]{Fact}
\newtheorem{proposition}[theorem]{Proposition}

\theoremstyle{definition}
\newtheorem{definition}[theorem]{Definition}
\newtheorem{condition}[theorem]{Condition}

\theoremstyle{remark}

\newtheorem{example}[theorem]{Example}

\newcommand{\R}{\mathbb{R}}
\newcommand{\Z}{\mathbb{Z}}
\newcommand{\N}{\mathbb{N}}

\newcommand{\Proba}[1]{\mathbb{P}\left( #1 \right)}
\newcommand{\norm}[2]{\left\Vert #1 \right\Vert_{#2}}

\newcommand{\scal}[2]{\left<#1,#2\right>}

\DeclareMathOperator{\cross}{Cross}
\DeclareMathOperator{\var}{Var}
\DeclareMathOperator{\cov}{Cov}

\title{Shadow and percolation II: discrete and continuous landscapes with correlations}

\author{David Vernotte}

\begin{document}
\maketitle

\begin{abstract}
    In this paper we consider a discrete or continuous landscape with correlations and we consider a source of light (a sun) at infinity emitting parallel rays of light making a slope $\ell\in \R$ with the horizontal plane. Depending on the value of $\ell$ some portions of the landscape may be lit by the sun or be in the shadow. Under some assumptions, we show  that if $\ell>0$ is big enough there exists a giant component of light, whereas if $\ell>0$ is small enough, and we are in the discrete case, then there exists a giant component of shadow. We relate this problem to the study of the percolation properties of a new random planar field.
\end{abstract}

\tableofcontents

\section{Introduction}

\begin{figure}
    \centering
    \includegraphics[width=0.8\textwidth]{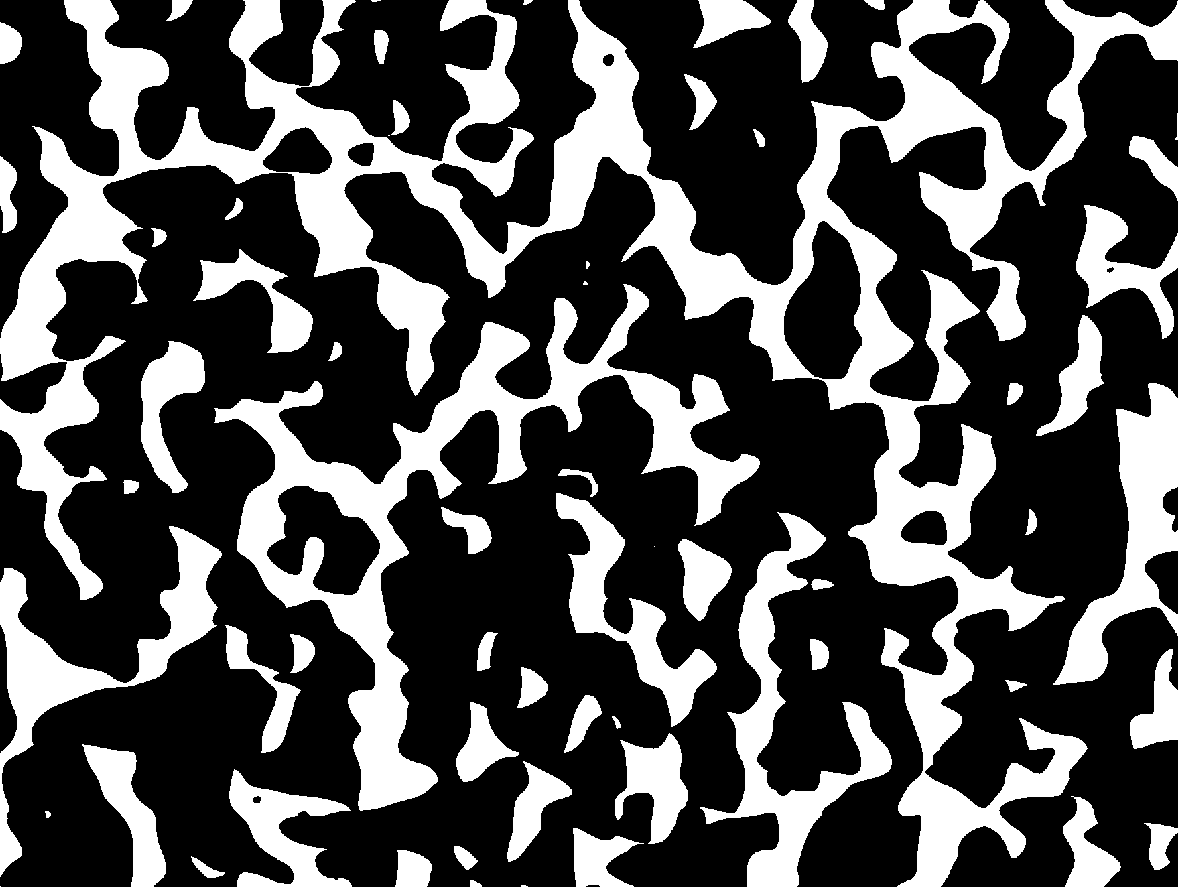}
    \caption{Illustration of $\{\alpha^f>\ell\}$ when $f$ is the Bargmann-Fock field and $\ell=0.3$. In black, the points $x$ such that $\alpha^f(x)>\ell$, in white the points $x$ such that $\alpha^f(x)\leq \ell$.}
    \label{fig:intro}
\end{figure}
In this paper, we study the percolation properties of the shadow of a planar random field. Given a function $f : \R^2\to \R$, we interpret $f$ as a height function describing a landscape of mountains. This landscape may be interpreted as a surface in $\R^3$ given by the graph of $f$. Given a parameter $\ell\in \R_+$, we consider a source of light located at infinity in the direction $e_1=(1,0)$ and emitting parallel rays of lights that make a slope $\ell$ with the horizontal plane $\R^2\times \{0\}.$ We ask the question of understanding and describing the set of points lit by the sun and the set of points that are in the shadow of some higher peak. More formally, given a function $f \in \mathcal{C}^1 (\R^2,\R)$ we introduce the field $\alpha^f$ by
\begin{equation}
    \label{eq:alpha_intro}
    \forall x\in \R^2,\ \alpha^f(x) := \sup_{r\in \mathbb{R}_+^*}\frac{f(x+re_1)-f(x)}{r} \in \overline{\R}.
\end{equation}
When $f$ is a random function, then $\alpha^f$ also becomes random. It appears that understanding which points are in the shadow when the slope is $\ell$ is equivalent to understanding the set
\begin{equation}
    \label{eq:level_set_intro}
    \{\alpha \geq \ell\} = \{x\in \R^2\ |\ \alpha^f(x)\geq \ell\}.
\end{equation}
However it follows from \eqref{eq:alpha_intro} that even when $f$ has a very simple law, the field $\alpha^f$ does not present many of the nice properties usually found in percolation models, such as invariance by rotation of the law, independence, duality of $\alpha^f$ and $-\alpha^f$, FKG property etc. In a previous paper the author studied a discrete version of this model where the landscape presented independence. In this paper we extend the model and the result to a broader class of landscapes, including continuous landscapes with correlations and discrete landscapes with correlations. However, while the previous paper dealt with general laws (as soon as there was independence between sites), in this paper we restrict our attention to Gaussian fields (but with correlations between sites). Indeed, such continuous and correlated fields have become a subject of interest in percolation, see \cite{Ale96}, \cite{BG17}, \cite{RV19}, \cite{MV20}, \cite{Sev21}, \cite{DCRRV23} for an incomplete list of references. First, we present how to construct such a continuous random Gaussian field $f : \R^2 \to \R$. We first introduce some notations.
\begin{definition}
A planar white noise $W$ is a centered Gaussian field indexed by functions of $L^2(\R^2)$ such that for any $\varphi_1,\varphi_2 \in L^2(\R^2)$ we have
\begin{equation}
    \mathbb{E}[W(\varphi_1)W(\varphi_2)]=\int_{\R^2}\varphi_1(x)\varphi_2(x)dx.
\end{equation}
\end{definition}
We refer to \cite{Jan97} for a construction and properties of the white noise. Given a function $q\in L^2(\R^2)$ such that $q(-x)=q(x)$, we set
\begin{equation}
    f:= q\ast W,
\end{equation}
where $\ast$ denotes convolution (that is, $f(x)=W(q(x-\cdot))$). Then, the field $f : \R^2 \to \R$ is a planar centered Gaussian field that is stationary (that is, the law of $(f(x))_{x\in \R^2}$ is the same as the law of $(f(x+y))_{x\in \R^2}$ for any $y\in \R^2$).
Moreover, a simple computation shows that
$$\mathbb{E}[f(x)f(y)]=(q\ast q)(x-y).$$
We make some assumptions on $q$.
\begin{assumption}
\label{a:q}
There exists some $\beta>1$ such that the following holds.
\begin{itemize}
    \item (Regularity) The function $q$ is in $\mathcal{C}^4(\R^2)$ and for any $\alpha=(\alpha_1,\alpha_2)\in \mathbb{N}^2$ with $\alpha_1+\alpha_2\leq 4$ the function $\partial^\alpha q$ is in $L^2(\R^2)$.
    \item (Symmetry) The function $q$ verifies $q(-x)=q(x)$.
    \item (Decay of correlations) There exists a constant $C>0$ such that for all $\alpha=(\alpha_1,\alpha_2)\in \mathbb{N}^2$ with $\alpha_1+\alpha_2\leq 2$ and for any $\norm{x}{}\geq 1$,
    $$|\partial^\alpha q(x)|\leq \frac{C}{\norm{x}{}^\beta}.$$
\end{itemize}
\end{assumption}
In the following we assume that $f=q\ast W$ where $q$ satisfies Assumption \ref{a:q} for some $\beta>1$. The regularity assumption allows us to see $f$ as a twice differentiable function from $\R^2$ to $\R$. The symmetry assumption guarantees the fact that the law of $f$ is stationary. Finally the assumption on the decay of correlations allows to get quasi-independence for the field $f$ (in particular this assumption implies ergodicity of the law of the field with respect to the flow of translations).
\begin{example}
An important example is the Bargmann-Fock field. If one takes $q(x)=\sqrt{\frac{2}{\pi}}e^{-\norm{x}{}^2}$ then the field $f=q\ast W$ verifies
$$\mathbb{E}[f(x)f(y)]=e^{-\frac{1}{2}\norm{x-y}{}^2},$$
and we say that $f$ is a planar Bargmann-Fock field. A Bargmann-Fock field also admits another construction which is as follows: let $(a_{i,j})_{i,j\geq 0}$ be a collection of independent standard Gaussian random variables, then the function
$$g(x)=e^{-\norm{x}{}^2/2}\sum_{i,j\geq 0}a_{i,j}\frac{x_1^ix_2^j}{\sqrt{i!j!}}$$
also has the law of the Bargmann-Fock field (it has the same covariance function as previously). We comment that from the expression of $q$, the Bargmann-Fock field satisfies Assumption \ref{a:q} for any $\beta>1$.
\end{example}
Given such a Gaussian field $f$, then the field $\alpha^f$ is defined by \eqref{eq:alpha_intro}.
One of the main result of this paper is the following.
\begin{theorem}
\label{thm:continuous_l_grand}
Let $f=q\ast W$ where $q$ satisfies Assumption \ref{a:q} for some $\beta>\frac{5}{2}$. There exists $\ell_1\in ]0,\infty[$ (depending on $q$) such that, for any $\ell>\ell_1$, almost surely, the set $\{\alpha^f \leq \ell\}$ contains a unique unbounded connected component and the set $\{\alpha^f \geq \ell\}$ does not contain any unbounded connected component.
\end{theorem}
This result may be interpreted as a first step towards proving that $\alpha^f$ presents a phase transition that is non-degenerated. Indeed one would also expect that $\{\alpha^f>\ell\}$ contains an unbounded component when $\ell>0$ is small enough. It appears that the regime of small $\ell>0$ is much harder to understand. We prove such a result only for a discrete version of $\alpha^X$. This is the content of the second result of this paper.
More precisely, we denote by $(X_u)_{u\in \Z^2}$ the discrete field defined as,
\begin{equation}
    \label{eq:defX}
    \forall u\in \Z^2,\ X_u := f(u)
\end{equation}
We make additional assumptions on the law of $(X_u)_{u\in \Z^2}$.
\begin{assumption}
\label{a:q2}
We say that $X$ follows Assumption \ref{a:q2} for some parameters $(R,\delta)\in \N^*\times \R_+^*$  if the two following conditions are satisfied.
\begin{itemize}
    \item For any $u\in \mathbb{Z}^2$ we have $\var(X_u)=\mathbb{E}[X_u^2]=1.$ 
    \item We have, $$\sum_{u\in \Z^2,\ \norm{u}{\infty}\geq R}|\cov(X_u,X_0)| <\delta.$$
\end{itemize}
\end{assumption}
Observe that if $f=q\ast W$ satisfies Assumption \ref{a:q} for some $\beta>4$, then we have $|\cov(X_u,X_0)| =O\left(\norm{u}{}^{2-\beta}\right)$ and therefore, $$\sum_{u\in \Z^2,\ \norm{u}{\infty}\geq R}\cov(X_u,X_0)||=O\left(R^{4-\beta}\right).$$ This implies, that for any $\delta>0$, by replacing $f(x)$ by $\lambda_1f(\lambda_2 x)$ with a good choices of $\lambda_1,\lambda_2>0$ (depending on $\delta$), then Assumption \ref{a:q2} is verified with parameters $(1,\delta)$.

We introduce a discrete version of $\alpha^f$ which we call $\alpha^X : \Z^2 \to \overline{\R}$ which is defined by
\begin{equation}
    \label{def:alpha_intro_discrete}
    \forall u\in \Z^2,\ \alpha^X(u) := \sup_{r\in \mathbb{N}^*}\frac{X_{u+re_1}-X_u}{r} \in \overline{\R} 
\end{equation}

Our second theorem is as follows.
\begin{theorem}
\label{thm:discrete_correlated}
Let $f=q\ast W$, $X$ defined as in \eqref{eq:defX} and $\alpha_X$ defined by \eqref{def:alpha_intro_discrete}. In the context of site percolation on $\Z^2$, the following holds.
\begin{itemize}
    \item If $q$ satisfies Assumption \ref{a:q} for some $\beta>\frac{5}{2}$, then there exists $\ell_1\in ]0,\infty[$ such that, for any $\ell>\ell_1$, almost surely, the set $\{\alpha^X\leq \ell\}$ contains a unique infinite cluster, and the set $\{\alpha^X \geq \ell\}$ only contains finite clusters.
    \item There exists an absolute constant $\delta_0>0$ such that, if $q$ verifies Assumption \ref{a:q} for some $\beta>16$ and Assumption \ref{a:q2} for $(1,\delta_0)$, then there exists $\ell_2\in ]0,\infty[$ such that, for any $\ell<\ell_2$, almost surely, the set $\{\alpha^X>\ell\}$ contains a unique infinite cluster and the set $\{\alpha^X<\ell\}$ only contains finite clusters.
\end{itemize}
\end{theorem}

\paragraph{Comments on Theorem \ref{thm:continuous_l_grand} and Theorem \ref{thm:discrete_correlated}.} First, we observe that both Theorems \ref{thm:continuous_l_grand} and \ref{thm:discrete_correlated} aim to understand the phase transition of the new percolation model associated to $\alpha^f$ and $\alpha^X$. We first observe that we do not prove an equivalent of the second item of Theorem \ref{thm:discrete_correlated} for continuous field. Indeed, in the discrete case our methods for studying the regime of small $\ell$ rely on Assumption \ref{a:q2} which is an Assumption on small scales and that does not generalize well to the context of continuous fields. It is an open question to find conditions under which this second item of Theorem \ref{thm:discrete_correlated} could be generalised to the continuous setting. Another remark, is that although we stated Theorem \ref{thm:discrete_correlated} for a discretization on $\Z^2$, the same proof would apply for a discretization on other lattices such as the triangular lattice or the union-jack lattice. In particular, on the triangular lattice, it is an open question to understand whether one may take $\ell_1=\ell_2$ in Theorem \ref{thm:discrete_correlated}. This problem is related to the understanding of the sharpness of the phase transition. However, most of the usual methods either use the FKG inequality (see for instance \cite{Kes80}, \cite{BR06}, \cite{BDC12}) or strong symmetry assumption on the field (see for instance \cite{Riv21}, \cite{MRV20}). But it appears that our field $\alpha^f$ lacks all those nice properties.

\paragraph{Strategy of proof}
Both proofs of Theorem \ref{thm:continuous_l_grand} and Theorem \ref{thm:discrete_correlated} share some ideas. Indeed, the two proofs rely on building approximations of $\alpha^f$ and $\alpha^X$ that are well-behaved. This is crucial since one may see from \eqref{eq:alpha_intro} that even if $f$ had finite range correlations this would not be the case for $\alpha^f$. In Section \ref{sec:2} we construct such good local approximations of $\alpha^f$ with have finite range correlations. We also prove that the field $\alpha^f$ is continuous and takes values in $\R$ (not $\overline{\R}$). These finite range approximations will be useful when we use a renormalization scheme. Indeed, in Section \ref{sec:renorm} we present a technical tool which is known as a \textit{renormalization scheme}. The objective is to develop an argument showing that if the field $\alpha^f$ behaves well in boxes of a certain fixed size, then it will continue to behave well in bigger boxes. This renormalization argument has already been used before in a variety of arguments, including percolation arguments (see \cite{Szn12}, \cite{ARS14}, \cite{DCGRS23}, \cite{DCRRV23} for instance). In Section \ref{sec:4}, we use these tools with classical gluing constructions (which we recall in Appendix \ref{sec:appendixA}) to conclude the proof of Theorem \ref{thm:continuous_l_grand}. In Section \ref{sec:5}, we provide the proof of Theorem \ref{thm:discrete_correlated}. However contrary to the proof of Theorem \ref{thm:continuous_l_grand} we need additional care. Indeed, if we applied directly our argument as is, we would prove $\ell_2>-\infty$ instead of $\ell_2>0$. Instead we crucially use the definition of $\alpha^X$ (see \eqref{def:alpha_intro_discrete}) and we present an argument to control the probability that a collection of a random variable following a Gaussian distribution is ordered according to some particular permutation. These estimates may have independent interest and make it possible to conclude the proof of Theorem \ref{thm:discrete_correlated}.

\paragraph{Acknowledgements: } I am very grateful to my PhD advisor Damien Gayet who first presented me this model and offered me to study it. I am also grateful to him for reviewing a first version of this paper.

\section{Preliminaries}
\label{sec:2}
In this section, we introduce some tools we use in several places of the paper and we collect several intermediate results. These tools notably include approximations of the field $\alpha^f$ which have finite range correlations. More precisely we make the following definition.
\begin{definition}
\label{def:approximations}
Let $(\Omega, \mathcal{A}, \mathbb{P})$ be a probability space. A random variable on $(\Omega, \mathcal{A}, \mathbb{P})$ taking values in $\mathcal{F}(\R^2,\R)$ (the space of functions from $\R^2$ to $\R$) is called a \textit{random field}. Let $a,b>0$. Let $g$ be a random field. We say that $g$ admits a \textit{$(a,b)$-sequence of finite range approximations} if there exists a sequence $(g_R)_{R\in \N^*}$ of random fields together with positive constants $c,C$ such that:
\begin{itemize}
    \item For any $R\in \N^*$ and $D_1,D_2\subset \R^2$ with $d(D_1,D_2)\geq R$ then the two collections $(g_R(x))_{x\in S_1}$ and $(g_R(x))_{x\in S_2}$ are independent.
    \item For any $\varepsilon>0$, and for any $R\in \N^*$,
    \begin{equation}
        \sup_{x\in \R^2}\Proba{\sup_{ x+[0,1]^2}|g-g_R|>\varepsilon}\leq Ce^{-c\varepsilon^a R^b}.
    \end{equation}
\end{itemize}
\end{definition}

 Such finite range approximations were introduced to study percolation problems related to Gaussian field (see \cite{MV20}, \cite{Sev21}, \cite{MRV20}, \cite{DRRV23} for instance). For the convenience of the reader we briefly present how to construct these approximations for the Gaussian field $f.$ Consider a smooth function $\chi : \R^2 \to [0,1]$ having the following properties
\begin{itemize}
    \item $\chi$ is in $\mathcal{C}^\infty(\R^2).$
    \item If $\norm{x}{}\leq \frac{1}{4}$, then $\chi(x)=1.$
    \item If $\norm{x}{}\geq \frac{1}{2}$, then $\chi(x)=0.$
    \item For all $x\in \R^2$, $\max(|\chi(x)|, \norm{\nabla \chi(x)}{}) \leq 10.$
    \item The function $\chi$ is radial (that is, $\chi(x)$ only depends on $\norm{x}{}$).
\end{itemize}
Given some parameter $R\geq 1$, we define the function $\chi_R$ by
\begin{equation}
    \chi_R(x):= \chi\left(\frac{x}{R}\right).
\end{equation}
Recall that the Gaussian field $f$ is defined as
$$f=q\ast W,$$
where $W$ is the white noise.
We define the $R$-truncature of $f$ by
\begin{equation}
    \label{eq:def_fR}
    f_R := (q\chi_R)\ast W.
\end{equation}
It is easy to check that $f_R$ is also a continuous planar centered Gaussian field which is $R$-correlated. This means that if $D_1$ and $D_2$ are two subsets of $\R^2$ at distance at least $R$ from one each other, then the random collections $(f_R(x))_{x\in D_1}$ and $(f_R(x))_{x\in D_2}$ are independent. Moreover, since $\chi_R$ converges to $1$ as $R$ goes to infinity, it is expected that $f_R$ is a good local approximation of $f$. This is the case as stated in the following proposition.

\begin{proposition}[see \cite{MV20}, \cite{Sev21}]
\label{prop:severo}
Let $f=q\ast W$ where $q$ satisfies Assumption \ref{a:q} for some $\beta >1$. There exist constants $C,c>0$ (depending on $q$) such that the following holds.
For any $R\geq 1$ and any $\varepsilon>0$ we have
\begin{align}
    \Proba{\sup_{[0,1]^2}|f-f_R|>\varepsilon}\leq Ce^{-c\varepsilon^2R^{2\beta-2}}, \\
    \Proba{\sup_{[0,1]^2}\norm{\nabla f-\nabla f_R}{}>\varepsilon}\leq Ce^{-c\varepsilon^2R^{2\beta-2}}.
\end{align}
\end{proposition}
Stationarity and Proposition \ref{prop:severo} show that if $\beta>1$ then $(f_R)_{R\in \N^*}$ is a $(2,2\beta-2)$-sequence of finite range approximations of $f$. In the following we aim to build a sequence of finite range approximation for the field $\alpha^f.$ First we note that Proposition \ref{prop:severo} is a consequence of the Borell-TIS inequality for Gaussian fields which we recall for the convenience of the reader.
\begin{proposition}[see \cite{AT07} for instance]
\label{prop:BORELL_TIS}
    Let $g$ be a centered Gaussian field on $\R^2$. Let $T\subset \R^2$. If $\sup_{x\in T} |g(x)|$ is finite almost surely then the following statements are verified:
    \begin{itemize}
        \item $\sigma_T^2:=\sup_{x\in T}\mathbb{E}[g(x)^2]<\infty,$
        \item $m_T:= \mathbb{E}[\sup_{x\in T}|g(x)|]<\infty,$
        \item $\forall u\in \R_+,\ \Proba{\sup_{x\in T}|g(x)|>m_T+u} \leq \exp\left(-\frac{u^2}{2\sigma_T^2}\right).$
    \end{itemize}
\end{proposition}
Other than Proposition \ref{prop:severo}, we also use Borell-TIS inequality in the form of the following statement.
\begin{proposition}
\label{prop:BORELL_TIS2}
    Assume that $f=q\ast W$, where $q$ satisfies Assumption \ref{a:q} for some $\beta >0$. There exists a constant $c>0$ such that for any $\alpha=(\alpha_1,\alpha_2)\in \mathbb{N}^2$ with $\alpha_1+\alpha_2\leq 2$, and for any $R\in [1,\infty]$,
    \begin{equation}
        \forall u\in \R_+,\ \Proba{\sup_{x\in [0,1]^2}{|\partial^\alpha f_R(x)|}>u}\leq \frac{1}{c}e^{-cu^2},
    \end{equation}
    with the convention $f_\infty=f.$
\end{proposition}
\begin{proof}
To ease notations, we do the proof for $\alpha=(0,0)$ but it is straightforward to extend the proof to $\alpha$ with $|\alpha|=\alpha_1+\alpha_2\leq 2.$
Let $T=[0,1]^2$. Let $R\in [1,\infty]$. We denote by $\sigma_T(R)$ and $m_T(R)$ the quantities,
$$\sigma_T(R) := \sup_{x\in T}\mathbb{E}\left[f_R(x)^2\right],$$
$$m_T(R) := \mathbb{E}\left[\sup_{x\in T}|f_R(x)|\right].$$
We observe that almost surely, $f_R$ is continuous on $T$ which is compact. This implies that $\sup_T|f_R|$ is almost surely finite and by Proposition \ref{prop:BORELL_TIS} we find,
\begin{equation}
    \forall R\in [1,\infty],\ \forall u\in \R_+,\ \Proba{\sup_{x\in [0,1]^2}|f_R(x)|>m_T(R)+u} \leq \exp\left(\frac{-u^2}{2\sigma_T(R)^2}\right).
\end{equation}
Now we observe that $\sup_{t\in [0,1]^2}|f(x)-f_R(x)|\xrightarrow[R\to \infty]{}0$ almost surely and in the $L^2$ sense (this follows for instance from Proposition \ref{prop:severo}). Therefore, we have $m_T(R)\xrightarrow[R\to \infty]{}m_T(\infty)$ and $\sigma_T(R)\xrightarrow[R\to \infty]{}\sigma_T(\infty).$
We may therefore find positive constants $m,\sigma$ such that
\begin{equation}
    \forall R\in [1,\infty],\ \forall u\in \R_+,\ \Proba{\sup_{x\in [0,1]^2}|f_R(x)|>m+u} \leq \exp\left(\frac{-u^2}{2\sigma^2}\right).
\end{equation}
Adjusting constants we get the conclusion.
\end{proof}
Before building a sequence of finite range approximations for $\alpha^f$ we provide some regularity results on $\alpha^f$. More precisely, we prove the following proposition
\begin{proposition}
    \label{prop:continuity}
    Let $f=q\ast W$ where $q$ satisfies Assumption \ref{a:q} for some $\beta>1.$
    Almost surely, the function $z\mapsto \alpha^f(z)$ is well defined, takes values in $\mathbb{R}_+^*$ and is continuous over $\R^2$.
\end{proposition}
Let us introduce some notations. Given a $\mathcal{C}^1(\R^2,\R)$ function $f$, given a point $z\in \R^2$ and $r\in \R_+$ we define
\begin{equation}
    \label{eq:def_tau}
    \tau_f(z,z+re_1):= \begin{cases}
    \frac{f(z+re_1)-f(z)}{r}& \text{ if }r>0, \\
    \scal{\nabla f(z)}{e_1} &\text{ if }r=0.
\end{cases}
\end{equation}
For $z=(x,y)\in \R^2$ we see that definition \eqref{eq:alpha_intro} of $\alpha^f$ is equivalent to the following,
\begin{equation}
    \alpha^f(z)=\alpha^f(x,y) = \sup_{r\in \R_+}\tau_f(z,z+re_1).
\end{equation}
We first prove the following technical result.
\begin{lemma}
    \label{lemma:technical}
    Let $f=q\ast W$, where $q$ satisfies Assumption \ref{a:q} for some $\beta >1$.
    Almost surely the following holds:
    \begin{enumerate}
        \item For any $(u,v)\in \Z^2$ there exists a random positive constant $C$ such that for all $(x,y)\in (u,v)+[0,\infty[\times [0,1]$, $|f(x,y)|\leq C\sqrt{x-u+1}.$
        \item For any $z\in \R^2$ there exists $r>0$ such that $\tau_f(z,z+re_1)>0.$
        \item For any $z\in \R^2$ the set $R(z) := \{r\in [0,\infty[, \alpha^f(z)=\tau_f(z,z+r(z)e_1)\}$ is non empty. Moreover if $r(z) := \inf R(z)$ then $r(z)\in R(z)$.
        \item The function $\alpha^f : \R^2\to \R$ is lower-semi-continuous (that is $\forall z\in \R^2, \liminf_{z_n\to z}\alpha^f(z_n)\geq \alpha^f(z).$) 
        \item For any $(u,v)\in \Z^2$, there exists a random positive constant $C$ such that for all $(x,y)\in (u,v)+[0,1]^2$, then $r(z)\leq C.$
    \end{enumerate}
\end{lemma}
\begin{proof}
    The first point is an application of the Borell-TIS inequality. Let $T_k = (k,0)+[0,1]^2$. Denote by $\mathcal{A}_k$ the event
    $$\mathcal{A}_k = \{\sup_{t\in T_k}|f(t)|>\sqrt{k}\}.$$
    By stationarity and applying Proposition \ref{prop:BORELL_TIS2}, we see that
    $$\sum_{k\geq 0}\mathbb{P}(\mathcal{A}_k)<\infty.$$
    Therefore, by the Borel-Cantelli lemma, we see that there exists an event $\mathcal{G}$ of probability $1$ under which there exists (a random) $k_0\geq 0$ such that for all $k\geq k_0$ we have $\sup_{t\in T_k}|f(t)|\leq \sqrt{k}$. This implies that
    $$\forall (x,y)\in [k_0,\infty[\times [0,1], |f(x,y)|\leq \sqrt{x+1}.$$
    Since $f$ is continuous on $\R^2$ and since $[0,k_0]\times [0,1]$ is bounded we may find a (random) constant $C>0$ such that
    $$\forall (x,y)\in [0,k_0]\times [0,1], f(x,y)\leq C \leq C\sqrt{x+1}.$$
    Altogether, by replacing $C$ by $\max(C,1)$ we get that almost surely
    $$\exists C>0,\forall (x,y)\in [0,\infty[, |f(x,y)|\leq C\sqrt{x+1}.$$
    Let $(u,v)\in \R^2$ denote by $\mathcal{G}_{u,v}$ the event
    $$\mathcal{G}_{u,v}:= \{\exists C>0, \ \forall (x,y)\in (u,v)+[0,\infty[\times [0,1], |f(x,y)|\leq C\sqrt{x-u+1}\}.$$
    We have shown that $\mathcal{G}_{0,0}$ has probability $1$, and therefore (by stationarity) all events $\mathcal{G}_{u,v}$ have probability $1$. On the intersection $\mathcal{G}^{(1)} := \bigcap_{u,v\in \Z^2}\mathcal{G}_{u,v}$ we have the first element of the lemma.

    We now provide the proof of the second item. For $k,m\in \N$ let us introduce the following events
    $$\mathcal{B}_{k,m} := \left\{\inf_{t\in T_k}f(t)\geq m\right\},$$
    where we recall that $T_k=(k,0)+[0,1]^2.$
    By stationarity, for fixed $m$ then all $\mathcal{B}_{k,m}$ have the same positive probability. Therefore, by ergodicity, for all fixed $m$ there are infinitely many values of $k\in \N$ such that the event $\mathcal{B}_{k,m}$ occurs. Denote by $\mathcal{B}_m$ the event
    $$\mathcal{B}_m:= \bigcup_{k\in \N^*}\mathcal{B}_{k,m}.$$
    Then, all $\mathcal{B}_m$ have probability one. Consider $\mathcal{G}^{(2)}:= \bigcap_{m\in \N^*} \mathcal{B}_m$. Then on this event, let $z\in [0,1]^2$, take $m>\sup_{t\in [0,1]^2}|f(t)|\geq f(z)$, on the event $\mathcal{G}^{(2)}$ some $\mathcal{B}_{k,m}$ must occur, if we take any $z+ke_1\in T_k$ we have $f(z+ke_1)\geq m$ (by the definition of the event $\mathcal{B}_{k,m})$ this shows that $\tau_f(z,z+ke_1)>0.$ We have proven that on the event $\mathcal{G}^{(2)}$ we have
    $$\forall z\in [0,1]^2,\exists r\geq 0, \tau_f(z,z+re_1)>0.$$
    It is straightforward to generalize this for all $z\in \R^2$ by covering $\R^2$ by countably many boxes of the form $(u,v)+[0,1]^2$ and using the stationarity of the field.

    We now turn to the proof of the third item. We work under the events $\mathcal{G}^{(1)}\cap \mathcal{G}^{(2)}$. Let $z=(x,y)\in \R^2$, by the first item we have
    $$\tau_f(z,z+re_1)=\frac{f(z+re_1)-f(z)}{r}=O\left(\frac{1}{\sqrt{r}}\right)\xrightarrow[r\to \infty]{}0.$$
    However by the second point we have some $r_0\geq 0$ such that
    $$\tau_f(z,z+r_0e_1)>0.$$ This means that we have some $R_0$ fixed (depending on $z$) such that $$\alpha^f(z)=\sup_{0\leq r\leq R_0}\tau_f(z,z+re_1).$$ Since the function $r\mapsto \tau_f(z,z+re_1)$ is continuous and since $[0,R_0]$ is compact, we deduce that there exists some $r\in [0,R_0]$ such that $\alpha^f(z)=\tau_f(z+re_1)$. This shows that the set $R(z)$ is non empty. If we set $r(z)=\inf R(z)$, then by continuity of $r\mapsto \tau_f(z,z+re_1)$ we see that $\alpha^f(z)=\tau_f(z,z+r(z)e_1)$. This concludes the proof of the third item.

    For the fourth item, take any sequence $(z_n)_n$ that converges to $z\in \R^2$. By the third item, there exists some $r\geq 0$ such that $\alpha^f(z)=\tau_f(z,z+re_1)$. Since $z\mapsto \tau_f(z,z+re_1)$ is continuous, we see that
    $$\liminf_{n\to \infty}\tau_f(z_n,z_n+re_1)\geq \tau_f(z,z+re_1)=\alpha^f(z,z+re_1).$$ Since we have $\alpha^f(z_n)\geq \tau_f(z_n,z_n+re_1)$ we get the conclusion.

    For the fifth item, we do the proof for $(u,v)=(0,0)$ (by stationarity it is enough). Assume by contradiction that we have a sequence $z_n\in [0,1]^2$ such that $r(z_n)\xrightarrow[n\to \infty]{}+\infty$. We then argue that $\alpha^f(z_n)\xrightarrow[n\to \infty]{}0$. In fact, under the event $\mathcal{G}^{(1)}$ we have $$\alpha^f(z_n) \leq \frac{C\sqrt{r_n+2}-f(z_n)}{r_n} \leq \frac{C\sqrt{r_n+2}}{r_n}+\frac{\sup_{z\in [0,1]}|f(z)|}{r_n}\xrightarrow[n\to \infty]{}0.$$
    By compactness of $[0,1]^2$ we may assume that the sequence $(z_n)$ converges to some $z\in [0,1]^2$. By lower semi-continuity of $\alpha^f$ we get $\alpha^f(z)\leq \liminf \alpha^f(z_n)\leq 0$, this is in contradiction with the second item.
\end{proof}
We now prove Proposition \ref{prop:continuity}.
\begin{proof}[Proof of Proposition \ref{prop:continuity}.]
    By Lemma \ref{lemma:technical}, we already know that almost surely, $\alpha^f$ is well defined on $\R^2$ takes values in $\R_+^*$ and is lower semi continuous. It only remains to show that almost surely $\alpha^f$ is continuous. We will show that almost surely $\alpha^f$ is continuous on $]0,1[^2$. By contradiction assume there exists $z\in ]0,1[^2$ and a sequence $(z_n)_n$ of points of $[0,1]^2$ such that $z_n \xrightarrow[n\to \infty]{}z$ but $\alpha^f(z_n) \not\xrightarrow[n\to \infty]{}\alpha^f(z)$. Then since $\alpha^f$ is lower semi-continuous it must be the case that there exists $\varepsilon_0>0$ such that $\liminf_{n\to \infty}\alpha^f(z_n)\geq \alpha^f(z)+\varepsilon_0.$ Without loss of generality we may assume that $\alpha^f(z_n)\xrightarrow[n\to \infty]{}\ell$ with $\ell\in [\alpha^f(z)+\varepsilon_0,\infty].$ Recall the notation $r(z_n)$ introduced in Lemma \ref{lemma:technical}. By the fifth item of Lemma \ref{lemma:technical} we have a constant $C>0$ such that for all $n$, $0\leq r(z_n)\leq C$. Without loss of generality we may therefore assume that there exists $r\in [0,C]$ such that $r(z_n)\xrightarrow[n\to \infty]{}r.$ But then $\alpha^f(z_n)=\tau_f(z_n,z_n+r(z_n)e_1)\xrightarrow[n\to \infty]{}\tau_f(z,z+re_1)\leq \alpha^f(z)$ (the passage to the limit is possible thanks to the continuity of $(z,r)\mapsto \tau(z,z+re_1)$). This yields a contradiction since we would get $\alpha^f(z)+\varepsilon_0\leq \alpha^f(z).$
\end{proof}

In the following we introduce variants of the field $\alpha^f$ that are good local approximations of the field $\alpha$ but that are less correlated. There are two natural ways to define a $R$-truncature of the field $\alpha^f$. The first one consists in replacing the field $f$ by its truncated version $f_R$ in definition \eqref{eq:alpha_intro}. The other possibility, which is very natural given the nature of the model, consists in taking the supremum over $r\in [0,R]$ in definition \eqref{eq:alpha_intro}. We observe that individually each of these methods should give a good local approximation of the field $\alpha^f$. However, none of them is $R$-decorrelated. In fact, to obtain a good approximation of the field $\alpha^f$ which is $R$-decorrelated we combine the two methods. More formally, given a field $g : \R^2 \to \R$ and $R>1$ we define the field $(\alpha_R^g(z))_{z\in \R^2}$ by
\begin{equation}
    \label{eq:def_alpha_R_g}
    \forall z\in \R^2,\ \alpha_R^g(z):= \sup_{r\in [0,R]} \tau_g(z,z+re_1),
\end{equation}
where we recall that $\tau$ is defined by \eqref{eq:def_tau}. Now, given a field $f=q\ast W$, we will be interested in the following approximations of the field $\alpha^f$: $\alpha_R^f$, $\alpha^{f_R}$, $\alpha_R^{f_R}$ but also $\alpha_{R_1}^{f_{R_2}}.$

\begin{lemma}
\label{lemma:trunc_1}
Let $f=q\ast W$ where $q$ satisfies Assumption \ref{a:q} for some $\beta>1$. Let $b:=\frac{2\beta-2}{3}$. There exist constants $C,c>0$ such that for any $R_0\in [1,\infty]$ and for any $R\geq 1$, we have
\begin{equation}
    \label{eq:lemma_trunc_2}
    \forall \varepsilon\in ]0,1],\ \Proba{\exists z\in [0,1]^2,\ |\alpha_{R_0}^f(z)-\alpha_{R_0}^{f_R}(z)|\geq \varepsilon}\leq CR_0e^{-c\varepsilon^2R^{b}}.
\end{equation}
\end{lemma}
\begin{proof}
    In this proof, we denote by $Q_0$ the square $Q_0:=[0,1]^2$. More generally, given some parameter $r>0$ we denote by $Q_r$ the rectangle $Q_r := [0,r+1]\times [0,1]$. When $g : \R^2 \to \R$ is a function, we denote by $\norm{g}{Q_r} := \sup_{x\in Q_r}|g(x)|$. Let $R\geq 1$ and $R_0\in [1,\infty].$
    Let $z\in Q_0$. Denote by $r_0^f(z)$ (resp. $r_0^{f_R}(z)$) the smallest point in $[0,R_0]$ such that $\alpha^{f}_{R_0}(z)=\tau_{f}(z,z+r_0^f(z)e_1)$ (resp. $\alpha^{f_R}_{R_0}(z) =\tau_{f_R}(z,z+r_0^{f_R}(z)(e_1))$). These are well defined due to Lemma \ref{lemma:technical}.
    By the definition of $\tau$ (see \eqref{eq:def_tau}) and using the triangular inequality we obtain the following upper bound:
    \begin{align}
        \label{eq:first_method}
         \alpha^{f_R}_{R_0}(z)& =\tau_{f_{R}}(z,z+r_0^{f_R}(z)e_1)\nonumber\\
         &\leq \tau_f(z,z+r_0^{f_R}(z)e_1)+\frac{2\norm{f-f_{R}}{Q_{R_0}}}{r_0^{f_R}(z)}\nonumber\\
         &\leq \alpha^f_{R_0}(z)+\frac{2\norm{f-f_{R}}{Q_{R_0}}}{r_0^{f_R}(z)},
    \end{align}
    However, this estimate is only useful when $r_0^{f_R}(z)$ is not too small (in comparison to $\norm{f-f_R}{Q_{R_0}}$). Instead, we use another estimate which is more relevant for the regime of small $r_0^{f_R}(z)$. Assume that $r_0^{f_R}(z)\in [0,1]$, then
    \begin{align}
            \label{eq:second_method}
            \alpha_{R_0}^{f_R}(z)&=\tau_{f_{R}}(z,z+r(z)e_1) \nonumber\\
            &\leq \scal{\nabla f_{R}(z)}{e_1}+r(z)\norm{\nabla^2 f_R}{Q_1}  \nonumber\\
            &\leq \scal{\nabla f(z)}{e_1}+r(z)\norm{\nabla^2 f_R}{Q_1}+\norm{\nabla f-\nabla f_{R}}{Q_0} \nonumber\\
            & \leq \alpha^f_{R_0}(z)+r(z)\norm{\nabla^2 f_R}{Q_1}+\norm{\nabla f -\nabla f_{R}}{Q_0}.
    \end{align}
    By Proposition \ref{prop:severo} and by union bound
    \begin{equation}
        \Proba{\norm{f-f_R}{Q_{R_0}}\leq 1}\geq 1-CR_0e^{-cR^{2\beta-2}}.
    \end{equation}
    On the event, $\left\{\norm{f-f_R}{Q_{R_0}}\leq 1\right\}$ we have two possibilities. Either $r^{f_R}(z)\geq \norm{f-f_R}{Q_{R_0}}^{1/2}$ in which case (using \eqref{eq:first_method}) we have
    \begin{equation}
        \label{eq:result_first_method}
        \alpha_{R_0}^{f_R}(z)-\alpha_{R_0}^f \leq 2\norm{f-f_R}{Q_{R_0}}^{1/2}.
    \end{equation}
    Otherwise, $r^{f_R}(z)\leq \norm{f-f_R}{Q_{R_0}}^{1/2}\leq 1$, in which case (using \eqref{eq:second_method}) we have
    \begin{equation}
        \label{eq:result_second_method}
        \alpha_{R_0}^{f_R}(z)-\alpha_{R_0}^f \leq \norm{f-f_R}{Q_{R_0}}^{1/2}\norm{\nabla^2 f_R}{Q_1}+\norm{\nabla f-\nabla f_R}{Q_0}.
    \end{equation}
    Combining \eqref{eq:result_first_method} and \eqref{eq:result_second_method} we obtain
    
    \begin{equation}
        \label{eq:result_combine}
        \alpha_{R_0}^{f_R}(z)-\alpha_{R_0}^f(z) \leq \norm{f-f_R}{Q_{R_0}}^{1/2}(2+\norm{\nabla^2 f_R}{Q_1})+\norm{\nabla f-\nabla f_R}{Q_0}.
    \end{equation}
    
    We obtain a similar lower bound by swapping $\alpha_{R_0}^f$ and $\alpha_{R_0}^{f_R}$. The only difference being that $\norm{\nabla^2 f}{Q_1}$ appears instead of $\norm{\nabla^2 f_R}{Q_1}$. Ultimately since, the upper bound in \eqref{eq:result_combine} is independent from $z\in Q_0$. We conclude that on the event $\left\{\norm{f-f_R}{Q_{R_0}}\leq 1\right\}$ we have
    \begin{equation}
        \label{eq:result_combine_2}
        \norm{\alpha_{R_0}^{f_R}-\alpha_{R_0}^{f}}{Q_0} \leq \norm{f-f_R}{Q_{R_0}}^{1/2}(2+\norm{\nabla^2 f}{Q_1}+\norm{\nabla^2 f_R}{Q_1})+\norm{\nabla f-\nabla f_R}{Q_0}.
    \end{equation}
    Recall that $b=\frac{2\beta-2}{3}$. Let $\varepsilon\in ]0,1]$ We first conclude the proof in the case  $\varepsilon^2R^b>1.$ Indeed, in this case, if $\norm{\alpha_{R_0}^f-\alpha_{R_0}^{f_R}}{Q_0}$ is greater than $\varepsilon\in ]0,1[$ we are in at least one of the following cases
    \begin{enumerate}
        \item $\norm{f-f_R}{Q_{R_0}}>\frac{1}{16}\varepsilon^{2/3}R^{-b}$.
        \item or $\norm{\nabla f-\nabla f_R}{Q_0}>\varepsilon/2$,
        \item or $\norm{\nabla^2 f}{Q_1}>\varepsilon^{2/3}R^{b/2}-1$,
        \item or $\norm{\nabla^2 f_R}{Q_1}>\varepsilon^{2/3}R^{b/2}-1$.
    \end{enumerate}
    We observe in particular that $\left\{\norm{f-f_R}{Q_{R_0}}> 1\right\}$ implies that we are in the first case since $\varepsilon\leq 1$ and $R^{b}\geq 1.$ We also observe that $\varepsilon^{2/3}R^{b/2} = (\varepsilon^2R^{b})^{1/3}R^{b/6}\geq 1$.
    In order to control the probability of the first and second cases we may use Proposition \ref{prop:severo}, in order to control the probability of the third and fourth item, we may use Proposition \ref{prop:BORELL_TIS2}.
    Altogether there exist positive constants $C,c$ that depends only on $q$ such that,
    \begin{align}
        &\Proba{\norm{\alpha_{R_0}^f-\alpha_{R_0}^{f_R}}{Q_0}>\varepsilon}\\
        \leq \quad & CR_0e^{-cR^{2\beta-2}\varepsilon^{4/3}R^{-2b}}+Ce^{-cR^{2\beta-2}\varepsilon^2}+Ce^{-cR^{b}\varepsilon^{4/3}} \nonumber\\
        \leq \quad & CR_0e^{-c\varepsilon^2R^{b}}
    \end{align}
    This concludes the proof in the case $\varepsilon^2R^b\geq 1.$ Now if $0< \varepsilon^2R^b\leq 1$ then $CR_0e^{-c\varepsilon^2R^b}\geq CR_0e^{-c}\geq Ce^{-c}$. It is therefore enough to replace $C$ by $\max(C,e^c)$ to get the conclusion.
\end{proof}

\begin{lemma}
\label{lemma:trunc_2}
Let $f=q\ast W$ where $q$ satisfies Assumption \ref{a:q} for some $\beta>1$.
    There exist constants $C,c>0$ such that for any $R_0\in [1,\infty]$ and $R\geq 1$ the following holds,
    \begin{equation}
        \label{eq:lemma_tronc_2}
        \forall \varepsilon\in ]0,1],\ \Proba{\exists z\in [0,1]^2,\ |\alpha^{f_{R_0}}(z)-\alpha^{f_{R_0}}_{R}(z)|\geq \varepsilon}\leq Ce^{-c\varepsilon^2R^2}
    \end{equation}
\end{lemma}
\begin{proof}
    For convenience, we write $g=f_{R_0}$.
    On the event that there exists $z\in [0,1]^2$ such that $|\alpha^g(z)-\alpha^g_{R}(z)|\geq \varepsilon$ then it means that for this $z$ we have
    $$\sup_{r\geq R}\tau_g(z,z+re_1)\geq \sup_{0\leq r\leq R}\tau_g(z,z+re_1)+\varepsilon.$$
    This implies that we are in at least one of the following two cases
    \begin{enumerate}
        \item Either there exists $z\in [0,1]^2$ such that, $\sup_{r\geq R}\tau_g(z,z+re_1)\geq \varepsilon/2$.
        \item or there exists $z\in [0,1]^2$ such that $\sup_{0\leq r\leq R}\tau_g(z,z+re_1)\leq -\varepsilon/2$.
    \end{enumerate}
    We denote by $\mathcal{A}$ the event,
    $$\mathcal{A}:= \left\{\exists z\in [0,1]^2,\ \sup_{r\geq R}\tau_g(z,z+re_1)\geq \varepsilon/2\right\}.$$
    On the event $\mathcal{A}$, there exists $z\in [0,1]^2$ and $r\geq R$ such that $g(z+re_1)\geq g(z)+r\varepsilon/2$ we thus split the event $\mathcal{A}$ into two subcases:
    \begin{itemize}
        \item Either $\inf_{z\in [0,1]^2}g(z)\leq -\frac{1}{4}R\varepsilon$, this has probability of order $e^{-cR^2\varepsilon^2}$ where $c,C>0$ can be chosen independently from $R_0$ by Proposition \ref{prop:BORELL_TIS2},
        \item or there exists $z\in [0,1]^2$ and $r\geq R$ such that
        $$g(z+re_1)\geq \frac{1}{2}r\varepsilon-\frac{1}{4}R \varepsilon \geq \frac{1}{4}r\varepsilon.$$
    \end{itemize}
    However, this last subcase implies that there must exist some $k\geq \lfloor R\rfloor$ such that the following event occurs.
    $$\mathcal{A}_k := \left\{\sup_{z'\in [k,k+1]\times[0,1]}g(z')\geq k\varepsilon/2\right\}.$$ By stationarity and Proposition \ref{prop:BORELL_TIS2} we see that there are positive constants $C,c$ (independent from $R_0$) such that
    \begin{equation}
        \Proba{\mathcal{A}_k}\leq Ce^{-c\varepsilon^2k^2}
    \end{equation}
    hence by an union bound we find that,
    \begin{equation}
        \label{eq:event_A}
        \Proba{\mathcal{A}}\leq Ce^{-cR^2\varepsilon^2},
    \end{equation}
    where again $C,c>0$ are positive constant that do not depend on $R_0.$
    
    We denote by $\mathcal{B}$ the event,
    $$\mathcal{B}:= \left\{\exists z\in [0,1]^2,\ \sup_{0\leq r\leq R}\tau_g(z,z+re_1)\leq -\frac{\varepsilon}{2}\right\}.$$
    On the event $\mathcal{B}$, there exists $z\in [0,1]^2$ such that for all $0\leq r\leq R$ then $g(z+re_1)\leq g(z)-r\varepsilon/2$
    We now split the event $\mathcal{B}$ into two subcases:
    \begin{itemize}
        \item either we have $\sup_{z\in [0,1]^2}g(z)\geq \frac{1}{8}R\varepsilon$, this has probability at most $Ce^{-cR^2\varepsilon^2}$ where $C,c>0$ do not depend on $R_0$ by Proposition \ref{prop:BORELL_TIS2},
        \item or there exists $z\in [0,1]$ such that for all $R/2\leq r\leq R$, $g(z+re_1)\leq -R\varepsilon/8$. This implies $$\sup_{[R-1,R]\times [0,1]}|g|\geq \varepsilon R/8$$
        which has probability at most $Ce^{-cR^2\varepsilon^2}$ with $C,c>0$ independent from $R_0$ by stationarity and Proposition \ref{prop:BORELL_TIS2}.
    \end{itemize}
    This proves
    \begin{equation}
        \label{eq:event_B}
        \Proba{\mathcal{B}}\leq Ce^{-cR^2\varepsilon^2},
    \end{equation}
    where $C,c>0$ are constants independent from $R_0.$
    Together with \eqref{eq:event_A}, \eqref{eq:event_B} concludes the proof.
\end{proof}
\begin{corollary}
\label{cor:trunc_3}
Let $f=q\ast W$ where $q$ satisfies Assumption \ref{a:q} for some $\beta>1$. Let $b_0:=\min\left(\frac{2\beta-2}{3},2\right)$. For any $b\in ]0,b_0[$ there exist constants $c,C>0$ such that for all $R\geq 1,$
\begin{equation}
    \label{eq:trunc_3}
    \forall \varepsilon\in ]0,1],\ \Proba{\exists z \in [0,1]^2,\ |\alpha^f(z)-\alpha^{f_R}_R(z)|>\varepsilon}\leq Ce^{-cR^b\varepsilon^2}.
\end{equation}
In particular, with the terminology of Definition \ref{def:approximations}, the sequence $(\alpha^{f_{R/4}}_{R/4})_{R\geq 1}$ is a $(2,b)$-sequence of finite range approximations of $\alpha^f$.
\end{corollary}
\begin{proof}
We have by triangular inequality and union bound,
\begin{align*}
    &\quad \Proba{\norm{\alpha^f-\alpha^{f_R}_R}{[0,1]^2}>\varepsilon} \\
    \leq &\quad \Proba{\norm{\alpha^f-\alpha^{f}_R}{[0,1]^2}>\varepsilon/2} + \Proba{\norm{\alpha^f_R-\alpha^{f_R}_R}{[0,1]^2}>\varepsilon/2} \\
    \leq &\quad Ce^{-cR^2\varepsilon^2}+CRe^{-cR^{\frac{2\beta-2}{3}}\varepsilon^2} \\
    \leq & \quad CRe^{-cR^{b_0}\varepsilon^2.}
\end{align*}
Since $b\in ]0,b_0[$ then by adjusting constants we get
\begin{equation}
    \Proba{\norm{\alpha^f-\alpha^{f_R}_R}{[0,1]^2}>\varepsilon} \leq Ce^{-cR^b\varepsilon^2}.
\end{equation}
Where from the second to third line we used Lemma \ref{lemma:trunc_1} and Lemma \ref{lemma:trunc_2}.
Now, if $D_1,D_2$ are two subsets of $\R^2$ such that $d(D_1,D_2)>R$ then the two collections $(\alpha^{f_{R/4}}_{R/4}(z))_{z\in D_1}$ and $(\alpha^{f_{R/4}}_{R/4}(z))_{z\in D_2}$ are independent, thus the two conditions in Definition \ref{def:approximations} are verified.
\end{proof}

\section{A renormalization argument}
\label{sec:renorm}
In this section we present a technical renormalization argument. This argument is technical and has already seen uses in various domains, see for instance \cite{Szn12}, \cite{DCGRS23}, \cite{DCRRV23}, \cite{ARS14} for an incomplete list of references. 
Let $d\geq 2$ be a dimension parameter. Although we will only use the case $d=2$, it is not much harder do present this argument for a general dimension.
Let $\sigma= (\sigma_n)_{n\in \N}$ and $\mu=(\mu_n)_{n\in \N}$ be two sequences of positive integers. Let $\lambda_0\in \mathbb{N}^*$. We define a sequence $(\lambda_n)_{n\in \N}$ by
\begin{equation}
    \forall n\in \N,\ \lambda_{n+1}= \lambda_n \mu_n.
\end{equation}
We denote by $\mathbb{L}_n$ the lattice
\begin{equation}
    \mathbb{L}_n := \lambda_n \Z^d.
\end{equation}
For $u\in \mathbb{L}_n$ we define $\Lambda_n(u)\subset \mathbb{L}_{n-1}$ by 
\begin{equation}
   \Lambda_n(u) = (u+[0,\lambda_n[^d)\cap \mathbb{L}_{n-1}.
\end{equation}
We observe that $\Lambda_n(u)$ contains exactly $\left(\frac{\lambda_n}{\lambda_{n-1}}\right)^d=\mu_{n-1}^d$ points.
 Let $(\Omega, \mathcal{F}, \mathbb{P})$ be a probability space. Suppose we have several collections of events of $\mathcal{F}$.
\begin{itemize}
    \item (\textit{Seed events}) A collection of events $\mathcal{A}^{(0)}= (\mathcal{A}^{(0)}_u)_{u\in \mathbb{L}_0}\in \mathcal{F}^{\mathbb{L}_0}.$ 
    \item (\textit{Auxiliary events}) For each $n\in \N^*$, a collection of events $\mathcal{B}^{(n)} := (\mathcal{B}^{(n)}_u)_{u\in \mathbb{L}_n}  \in \mathcal{F}^{\mathbb{L}_n}$.
\end{itemize}
A renormalization scheme (of dimension $d$) associated to $(\lambda_0, \mu,\sigma)$ is the data of a collection of seed events and collections of auxiliary events for each $n\in \N^*$. We write $\mathcal{R}=(\mathcal{A}^{(0)}, (\mathcal{B}^{(n)})_{n\in \N^*})$ such a renormalization scheme.

Given a renormalization scheme $\mathcal{R}=(\mathcal{A}^{(0)}, (\mathcal{B}^{(n)})_{n\in \N^*})$, we recursively define the events $\mathcal{A}_u^{(n)}$ for $n\in \N^*$ and $u\in \mathbb{L}_n$ as follows:
\begin{equation}
    \label{eq:def_renorm}
    \mathcal{A}^{(n)}_u := \mathcal{B}^{(n)}_u \cap \bigcap_{\overset{u_1,u_2 \in \Lambda_n(u)}{\norm{u_1-u_2}{\infty}\geq 5\sigma_{n-1}  \lambda_{n-1}}} \mathcal{A}^{(n-1)}_{u_1}\cup \mathcal{A}^{(n-1)}_{u_2},
\end{equation}
where the second intersection is taken over all points $u_1,u_2$ in $\Lambda_n(u)$ such that there are distant of at least $5\sigma_{n-1} \lambda_{n-1}$ for the sup norm.
To understand why this definition of renormalization argument is interesting we introduce the following definitions.
\begin{definition}
    Let $m\geq 0$ and $u\in \mathbb{L}_m$. There exists a unique sequence $(u_k)_{k\geq m}$ such that
    \begin{itemize}
        \item $u_m=u$.
        \item $\forall k\geq m,\ u_k \in \mathbb{L}_k.$
        \item $\forall k\geq m,\ u_k\in \Lambda_{k+1}(u_{k+1}).$
    \end{itemize}
    This sequence is called the \textit{renormalization sequence} associated to $u\in \mathbb{L}_m$
\end{definition}
\begin{definition}
    Let $\mathcal{R}=(\mathcal{A}^{0}, (\mathcal{B}^{(n)})_{n\in \N^*})$ a renormalization scheme associated to $(\lambda_0,\mu,\sigma)$. Let $n,m\in \N$ such that $n\geq m$ and $u\in \mathbb{L}_m$. Let $(u_k)_{k\geq m}$ the renormalization sequence associated to $u$. The point $u$ is said to be \textit{good up to scale $n$} if all events $\mathcal{A}^{(k)}_{u_k}$ occur for $m\leq k\leq n$.
    A point $u\in \mathbb{L}_0$ is said to be \textit{$n$-good} if it is good up to scale $n$.
\end{definition}
In particular a point $u\in \mathbb{L}_0$ being $n$-good implies that the seed event $\mathcal{A}^{(0)}_u$ occurs and that all auxiliary events $\mathcal{B}^{(k)}_{u_k}$ occur for $1\leq k\leq n$.
The main objective of the renormalization scheme is to find long paths of $n$-good points in $\mathbb{L}_0$ with high probability. However we need the renormalization scheme to satisfy some conditions for this to hold.
\begin{condition}
    \label{cond:c1}
    We say that $(\mu,\sigma)$ verifies Condition \ref{cond:c1} if
    \begin{equation}
            \forall n\in \N,\ \mu_n\geq 100 \sigma_n \text{ and }\sigma_n \geq 2.
    \end{equation}
\end{condition}
We argue that under Condition \ref{cond:c1}, then given a nearest-neighbor path in $\mathbb{L}_m$ of points good up to scale $n$ then one may extract a nearest-neighbor path in $\mathbb{L}_{m-1}$ of points good up to scale $n$. This process can be repeated to ultimately extract a nearest-neighbor path in $\mathbb{L}_0$ vertices that are $n$-good.
\begin{lemma}
\label{lemma:der_renorm_extract_path}
Let $\mathcal{R}=(\mathcal{A}^{0}, (\mathcal{B}^{(n)})_{n\in \N^*})$ be a renormalization scheme associated to $(\lambda_0,\mu,\sigma)$. Assume that Condition \ref{cond:c1} holds.
Let $n,m\in \N$ be two integers such that $n\geq m\geq 1$. Let $\gamma = (u_1,u_2,\dots,u_k)$ be a nearest-neighbor path of points in $\mathbb{L}_m$ (that is $\norm{u_i-u_{i+1}}{1}=\lambda_m$ for $1\leq i \leq k-1$). Assume that all the points $u_i$ are good up to scale $n$. Then there exists $\tilde{\gamma}=(v_1,v_2,\dots, v_l)$ a nearest-neighbor path of points in $\mathbb{L}_{m-1}$ such that
\begin{itemize}
    \item $v_1\in \Lambda_m(u_1).$
    \item $v_l\in \Lambda_m(u_k).$
    \item All the points $v_i$ (for $1\leq i \leq l$) are good up to scale $n$.
\end{itemize}
\end{lemma}
\begin{proof}
    Recall that for $u\in \mathbb{L}_m$, then the box $\Lambda_m(u)$ is defined as
    $$\Lambda_m(u)= u+[0,\lambda_m[^d\cap \mathbb{L}_{m-1}.$$
    For $1\leq i \leq d$ we denote by $F_i^+(u)$ and $F_i^-(u)$ the  discrete faces of the box $\Lambda_m(u)$, that is
    \begin{align}
            F_i^-(u) &:= u+ ([0,\lambda_m[^{i-1}\times \{0\}\times [0,\lambda_m[^{d-i})\cap \mathbb{L}_{m-1}. \\
    F_i^+(u)& := u+ ([0,\lambda_m[^{i-1}\times \{\lambda_m-\lambda_{m-1}\}\times [0,\lambda_m[^{d-i})\cap \mathbb{L}_{m-1}.
    \end{align}
    Observe that under Condition \ref{cond:c1}, then for any direction $1\leq i \leq d$, $$|F_i^+|=|F_i^-|=\left(\frac{\lambda_m}{\lambda_{m-1}}\right)^{d-1}=(\mu_{m-1})^{d-1}\geq 100^{d-1}(\sigma_{m-1})^{d-1}.$$
    Observe also that under the event $\mathcal{A}^{(m)}_u$ then the set of $v\in \Lambda_m(u)$ such that the event $\mathcal{A}^{(m-1)}_v$ does not occur is included in a box $B(u)$ of the form $B(u) = b_0(u)+[0,5\sigma_{m-1} \lambda_{m-1}]^d\cap \mathbb{L}_{m-1}$ for some $b_0(u)\in \Lambda_m(u)$.
    We have $$(5\sigma_{m-1})^{d-1}\leq \frac{1}{4} 100^{d-1}(\sigma_{m-1})^{d-1}.$$
    Therefore, on the event $\mathcal{A}^{(m)}_u$ then for any $1\leq i \leq d$ we may find subsets
    $\tilde{F}_i^+(u)\subset F_i^+(u)$ and $\tilde{F}_i^-(u) \subset F_i^-(u)$ such that
    \begin{itemize}
        \item $|\tilde{F}_i^+(u)|\geq \frac{3}{4}|F_i^+|$ and $|\tilde{F}_i^-(u)|\geq \frac{3}{4}|F_i^-|$, and
        \item any two points in $\bigcup_{1\leq i\leq d}\tilde{F}_i^-(u)\cup \tilde{F}_i^+(u)$ are connected by a nearest-neighbor path in $\Lambda_{m}(u)\setminus B(u).$
    \end{itemize}
    We may now complete the proof of Lemma \ref{lemma:der_renorm_extract_path}. In fact, given two points $u$ and $v$ that are nearest neighbors in $\mathbb{L}_m$ and that are good up to scale $n$, then the two boxes $u+[0,\lambda_m]^d$ and $v+[0,\lambda_m]^d$ share a side (which corresponds to some $u+F^\pm_i$ and to some $v+F^\mp_i.$ Since $|\tilde{F}^\pm_i(u)|\geq \frac{3}{4}|F_i^\pm|$ and $|\tilde{F}^\mp_i(v)|\geq \frac{3}{4}|F_i^\mp|$ and $3/4+3/4>1$ then one may find two vertices $x\in \tilde{F}^\pm_i(u)$ and $y\in \tilde{F}^\mp(v)$ that are adjacent in $\mathbb{L}_{m-1}.$ We may use this connection to connect all points of $\bigcup_{1\leq i\leq d}\tilde{F}_i^-(u)\cup \tilde{F}_i^+(u)\cup \tilde{F}_i^-(v)\cup \tilde{F}_i^+(v)$
    by paths that are orthogonally connected in $\mathbb{L}_{m-1}$ and that stay in $\Lambda_m(u)\setminus B(u)\cup\Lambda_m(v)\setminus B(v).$
    We have proven the case $k=2$ of the lemma but this argument is straightforward to extend to general $k\geq 2$.
\end{proof}
Now that we have seen how a renormalization scheme will be of interest we consider the problem of finding conditions under which the events $\mathcal{A}^{(n)}_u$ will have high probability. We introduce the following definition
\begin{definition}
\label{def:support}
Let $\mathcal{R}=(\mathcal{A}^{0}, (\mathcal{B}^{(n)})_{n\in \N^*})$ be a renormalization scheme associated to $(\lambda_0, \mu,\sigma)$. We associate to each event of $\mathcal{A}^{(0)}$ and $\mathcal{B}^{(n)}$ a \textit{formal support} which is a subset of $\R^d$.
\begin{itemize}
    \item For $u\in \mathbb{L}_0$, the formal support of the seed event $\mathcal{A}^{(0)}_u$ is defined as $u+[-\sigma_0 \lambda_0,\sigma_0 \lambda_0]^d$.
    \item For $n\in \N^*$ and $u\in \mathbb{L}_n$, the formal support of the auxiliary event $\mathcal{B}^{(n)}_u$ is defined as $u+[-\sigma_n \lambda_n,\sigma_n \lambda_n]^d.$
\end{itemize}
We extend this definition of formal support to all $\mathcal{A}^{(n)}_u$ for $n\in \N^*$ and $u\in \mathbb{L}_n$ as follows, the formal support of the event $\mathcal{A}^{(n)}_u$ is defined as the union of the formal supports of all events intervening in the recursive definition of $\mathcal{A}^{(n)}_u$ (see \eqref{eq:def_renorm}).
\end{definition}

We introduce the following condition
\begin{condition}
    \label{cond:c2}
    Let $\mathcal{R}=(\mathcal{A}^{(0)}, \mathcal{B}^{(n)})$ be a renormalization scheme associated to some $(\lambda_0,\mu,\sigma).$
    Let $\mathcal{C}$ be the reunion of the collection $\mathcal{A}^{(0)}$ of seed events and of the collections $\mathcal{B}^{(n)}$ of auxiliary events for all $n\in \N^*.$ For $D\subset \R^d$ we denote by $\mathcal{C}(D)$ the sub-collection of $\mathcal{C}$ obtained by only keeping events whose formal supports are included in $D$.
    We say that $\mathcal{R}$ verifies Condition \ref{cond:c2} if for any two subsets $D_1,D_2\subset \R$ then $D_1\cap D_2=\emptyset$ implies that the two collections $\mathcal{C}(D_1)$ and $\mathcal{C}(D_2)$ are independent.
\end{condition}
As a consequence of this condition, we show that two events $\mathcal{A}^{(n)}_u$ and $\mathcal{A}^{(n)}_v$ are independent if $u,v\in \mathbb{L}_n$ are far enough from one another.
\begin{lemma}
\label{lemma:indep_renorm}
Let $\mathcal{R}=(\mathcal{A}^{0}, (\mathcal{B}^{(n)})_{n\in \N^*})$ a renormalization scheme associated to some $(\lambda_0,\mu,\sigma)$. Assume that both Conditions \ref{cond:c1} and \ref{cond:c2} hold.
Let $n\in \N$. Let $u,v\in \mathbb{L}_n=\lambda_n\Z^d$. If $\norm{u_1-u_2}{\infty}\geq 5\sigma_n \lambda_{n}$ the two events $\mathcal{A}^{(n)}_u$ and $\mathcal{A}^{(n)}_{v}$ are independent.
\end{lemma}
\begin{proof}
We observe that by Definition \ref{def:support} and Condition \ref{cond:c2} it is enough to show that the formal support of $\mathcal{A}^{(n)}_u$ and $\mathcal{A}^{(n)}_u$ are disjoint.
    We argue that for each $n\in \N$ and $u\in \mathbb{L}_n$ the support of $\mathcal{A}_u^{(n)}$is included in
    $$u+[-2\sigma_n \lambda_n,2\sigma_n \lambda_n]^d.$$
    This property is true for $n=0$ by definition of the formal support of the seed events.
    If we assume that this property is true up to $n-1$, then if $u\in \mathbb{L}_n$ we have by \eqref{eq:def_renorm} that the formal support of $\mathcal{A}_u^{(n)}$ is included in the union of the formal support of $\mathcal{B}_u^{(n)}$ and the formal supports of the $\mathcal{A}^{(n-1)}_{v}$ for $v\in \Lambda_n(u)$.
    Therefore the formal support of $\mathcal{A}_{u}^{(n)}$ is included in
    $$u+\left([-\sigma_n \lambda_n,\sigma_n \lambda_n]^d \cup [-2\sigma_{n-1} \lambda_{n-1}, \lambda_n+2\sigma_{n-1} \lambda_{n-1}]^d\right).$$
    By Condition \ref{cond:c1}, then $\frac{\sigma_{n-1}\lambda_{n-1}}{\lambda_n} = \frac{\sigma_{n-1}}{\mu_{n-1}}\leq \frac{1}{100}.$ Moreover, we have $\sigma_n\geq 2$. Therefore,
    $$\lambda_n + 2\sigma_{n-1}\lambda_{n-1}\leq 2\sigma_n\lambda_n \text{ and } -2\sigma_{n-1}\lambda_{n-1}\geq -2\lambda_n\sigma_n.$$
    This implies that the formal support of $\mathcal{A}_{u}^{(n)}$ is included in
    $$u+[-2\sigma \lambda_n,2\sigma \lambda_n]^d,$$
    which finishes the induction.
    We now conclude the proof. Let $u,v$ be two vertices of $\mathbb{L}_n$ such that $\norm{u-v}{\infty}\geq 5\sigma_n \lambda_n$, then the formal supports of $\mathcal{A}_u^{(n)}$ and $\mathcal{A}^{(n)}_v$ do not intersect which implies independence.
\end{proof}

We now prove that under some additional conditions then the events $\mathcal{A}^{(n)}$ will have high probability. More precisely consider the following conditions.
\begin{condition}
    \label{cond:c3}
    We say that a sequence $\mu=(\mu_n)_{n\in \N}$ verifies Condition \ref{cond:c3} if
    \begin{equation}
        \sum_{n=0}^\infty \frac{\log_2(\mu_n)}{2^n}<\infty,
    \end{equation}
    where $\log_2$ denotes the logarithm in base $2$.
\end{condition}

\begin{condition}
    \label{cond:c4}
    Let $\mathcal{R}=(\mathcal{A}^{(0)}, \mathcal{B}^{(n)})$ be a renormalization scheme associated to some $(\lambda_0,\mu,\sigma).$ Let $\varepsilon>0.$
    We say that $\mathcal{R}$ satisfies Condition \ref{cond:c4} for $\varepsilon$ if
        \begin{equation}
    \label{eq:c3}
   \sup_{u\in \mathbb{L}_0}\Proba{\left(\mathcal{A}^{(0)}_u\right)^c} < \varepsilon \quad \text{ and}
\end{equation}
\begin{equation}
    \label{eq:c4}
    \forall n\in \N^*,\ \sup_{u\in \mathbb{L}_n}\Proba{\left(\mathcal{B}^{(n)}_u\right)^c} < 2^{-\frac{1}{\varepsilon}2^n}.
\end{equation}
\end{condition}

\begin{proposition}
\label{prop:renorm_high_prob}
Let $d\geq 2$. Let $\mu=(\mu_n)_n$ and $\sigma=(\sigma_n)_n$ be two sequences of positive integers such that Conditions \ref{cond:c1} and \ref{cond:c3} hold.
There exists $\varepsilon_0>0$ depending only on $d$, $\mu$ and $\sigma$ such that the following holds.
For any $\lambda_0\in \N^*$ and for any $\mathcal{R}=(\mathcal{A}^{(0)}, (\mathcal{B}^{(n)})_{n\in \N^*})$ a renormalization scheme of dimension $d$ associated to $(\lambda_0, \mu,\sigma)$, if $\mathcal{R}$ verifies Condition \ref{cond:c2} and Condition \ref{cond:c4} for $\varepsilon_0$ then,
\begin{equation}
    \label{eq:conclu_renorm}
    \forall n\in\N,\ \sup_{u\in \mathbb{L}_n} \Proba{\left(\mathcal{A}^{(n)}_u\right)^c}< 2^{-2^n}.
\end{equation}
\end{proposition}
\begin{proof}
Let $\mu=(\mu_n)_{n\in \N}$ and $\sigma=(\sigma_n)_{n\in \N}$ be such that Conditions \ref{cond:c1} and \ref{cond:c3} hold. Let $\lambda_0\in \N^*$ and $\mathcal{R}$ be a renormalization scheme associated to $(\lambda_0,\mu,\sigma)$ satisfying Condition \ref{cond:c2}.
 We may apply Lemma \ref{lemma:indep_renorm} and we see that any two events $\mathcal{A}^{(n-1)}_{u_1}$ and $\mathcal{A}^{(n-1)}_{u_2}$ are independent as soon as $\norm{u_1-u_2}{\infty}\geq 5\sigma_{n-1} L_{n-1}$.

For $n\in \N$ denote by $p_n$ the following:
\begin{equation}
    p_n := \sup_{u\in \mathbb{L}_n}\Proba{\left(\mathcal{A}_u^{(n)}\right)^c}.
\end{equation}
We aim to prove
\begin{equation}
    \forall n\in \N,\ p_n\leq 2^{-2^n}.
\end{equation}
Using the recursive definition of $\mathcal{A}^{(n)}$ (see \eqref{eq:def_renorm}), then by doing an union bound and applying the independence from Lemma \ref{lemma:indep_renorm} we obtain

\begin{equation}
    \label{eq:rec}
    \forall n\in \N^*,\ p_n \leq \sup_{u\in \mathbb{L}_n}\Proba{\left(\mathcal{B}^{(n)}_u\right)^c} + \left(\mu_{n-1}\right)^{2d}p_{n-1}^2.
\end{equation}
Unfortunately, it is not possible to directly show the conclusion by induction using \eqref{eq:rec}. One needs to modify a little the induction property.

We denote the logarithm in base $2$ by $\log_2$. Write
$$\Sigma_0 := \sum_{n=1}^\infty \frac{1+2d\log_2(\mu_{n-1})}{2^n}.$$
By Condition \ref{cond:c3}, $\Sigma_0$ is a positive constant depending only on $d$ and $\mu$.
Define a sequence $(a_n)_{n\in \N}$ such that $a_0:= 2^{\Sigma_0}$ and 
\begin{equation}
\label{eq:rec_sequence_a}
    \forall n\in \N^*,\ a_{n}:=\frac{1}{2}a_{n-1}^2(\mu_{n-1})^{-2d}.
\end{equation}
We claim the following fact
\begin{claim}
    \label{claim:a_sequence}
    $\forall n\in \N,\ 0\leq \log_2(a_n)\leq \log_2(a_0)2^n.$
\end{claim}
\begin{proof}
Taking the logarithm in \eqref{eq:rec_sequence_a} we get
$$\forall n\in \N^*,\ \log_2(a_n)=2\log_2(a_{n-1})-1-2d\log_2(\mu_{n-1}).$$
We divide by $2^n$ to obtain 
$$\forall n\in \N^*,\ \frac{\log_2(a_n)}{2^n}=\frac{\log_2(a_{n-1})}{2^{n-1}}-\frac{1+2d\log_2(\mu_{n-1})}{2^n}.$$
By cancellation we find
$$\forall n\in \N,\  \frac{\log_2(a_0)}{2^0}-\Sigma_0\leq \frac{\log_2(a_n)}{2^n}\leq \frac{\log_2(a_0)}{2^0}.$$
Since $\log_2(a_0)=\Sigma_0$ (by definition of $a_0$) we get the result.
\end{proof}

Choose $\varepsilon_0>0$ small enough (depending only on $d$, $\mu$) such that
\begin{equation}
    \varepsilon_0 < \min\left(\frac{1}{2+\Sigma_0}, \frac{1}{2a_0}\right).
\end{equation}
We will prove that if $\mathcal{R}$ verifies Condition \ref{cond:c4} for $\varepsilon_0$ then we have the conclusion. More precisely we will show by induction that
\begin{equation}
    \forall n\in \N,\ p_n \leq \frac{2^{-2^n}}{a_n}.
\end{equation}
Indeed, the property is true for $n=0$ by \eqref{eq:c3} and the fact that $\varepsilon_0<\frac{1}{2a_0}.$ If the property holds for $n-1$ then by \eqref{eq:rec}, \eqref{eq:rec_sequence_a} and \eqref{eq:c4} we have
\begin{equation}
    p_n \leq 2^{-\frac{1}{\varepsilon_0}2^n}+(\mu_{n-1})^{2d}\left(\frac{2^{-2^{n-1}}}{a_{n-1}}\right)^2\leq 2^{-\frac{1}{\varepsilon_0}2^n} + \frac{1}{2}\frac{2^{-2^n}}{a_n}.
\end{equation}
To conclude that $p_n\leq \frac{2^{-2^n}}{a_n}$ it is therefore enough to check that
$$2^{-\frac{1}{\varepsilon_0}2^n}\leq \frac{1}{2}\frac{2^{-2^n}}{a_n}.$$
Taking the logarithm, this condition is equivalent to
$$-\frac{1}{\varepsilon_0}2^n \leq -2^n-1-\log_2(a_n) \Leftrightarrow \varepsilon_0\leq\frac{1}{1+\frac{1}{2^n}+\frac{\log_2(a_n)}{2^n}}.$$
According to Fact \ref{claim:a_sequence}, we know that for $n\in \N^*$, $\frac{\log_2(a_n)}{2^n}\leq \log_2(a_0)=\Sigma_0.$
Moreover we also have $\frac{1}{2^n}\leq 1$. Therefore we get, $$\varepsilon_0<\frac{1}{2+\Sigma_0}\leq \frac{1}{1+\frac{1}{2^n}+\frac{\log_2(a_n)}{2^n}}.$$ This concludes the induction and we have $$\forall n\in\N,\ p_n \leq \frac{2^{-2^n}}{a_n}.$$
Since we have seen in Fact \ref{claim:a_sequence} that $\log_2(a_n)\geq 0$, then $a_n\geq 1$ and we get
$$\forall n\in\N,\ p_n \leq 2^{-2^n},$$
which is the conclusion of Proposition \ref{prop:renorm_high_prob}.
\end{proof}
We comment that a key element of the statement of Proposition \ref{prop:renorm_high_prob} is that $\varepsilon_0>0$ may be chosen independently from $\lambda_0$. As such, a natural way to use Proposition \ref{prop:renorm_high_prob} is to take $(\mu,\sigma)$ satisfying Conditions \ref{cond:c1} and \ref{cond:c3} and to work with a sequence of a renormalization schemes $(\mathcal{R}_{\lambda})_{\lambda\geq 1}$ such that $\mathcal{R}_\lambda$ is associated to $(\lambda, \mu, \sigma)$ (all $\mathcal{R}_\lambda$ share the same $(\mu,\sigma)$) and $\mathcal{R}_\lambda$ satisfies Condition \ref{cond:c2}. The strategy is then to check that Condition \ref{cond:c4} is verified asymptotically when $\lambda$ goes to infinity (that is we show that for every $\varepsilon>0$ then for $\lambda$ big enough, $\mathcal{R}_\lambda$ satisfies Condition \ref{cond:c4} for this $\varepsilon$). This will ensure that we may find $\lambda_0$ for which $\mathcal{R}_{\lambda_0}$ satisfies the conclusion of Proposition \ref{prop:renorm_high_prob}.

\section{Proof of Theorem \ref{thm:continuous_l_grand}}
\label{sec:4}
This section is dedicated to the proof of Theorem \ref{thm:continuous_l_grand}. This proof essentially relies on the renormalization argument developed in Section \ref{sec:renorm} together with the construction of the finite range approximations of the field $\alpha^f$ (see Definition \ref{def:approximations}). However we will also need to use some standard results of percolation. We begin by introducing a few notations.

\begin{definition}
Given a function $g : \R^2 \to \R$ and a level $\ell\in \R$, we denote by $\mathcal{E}_\ell(g)$ the set,
\begin{equation}
    \mathcal{E}_\ell(g) := \{x\in \R^2\ |\ g(x)+\ell\geq 0\}.
\end{equation}
\end{definition}
We recall the notion of crossing events.
\begin{definition}
\label{def:cross}
Let $\mathcal{R}=[a,b]\times [c,d]\subset \R^2.$ be a rectangle (with $a<b$ and $c<d$).
We denote by $\cross^h(\mathcal{R})$ the set of functions $g : \R^2 \to \R$ such that there exists a connected component of $\mathcal{E}_0(g)\cap \mathcal{R}$ that intersects both sides $\{a\}\times [c,d]$ and $\{b\}\times [c,d]$. Similarly we denote by $\cross^v(\mathcal{R})$ the set of functions $g : \R^2 \to \R$ such that there exists a connected component of $\mathcal{E}_0(g)\cap \mathcal{R}$ that intersects both sides $[a,b]\times \{c\}$ and $[a,b]\times \{d\}.$
\end{definition}
In the following we will freely make use of the following theorem.
\begin{theorem}
\label{thm:crossings}
Let $g : \R^2 \to \R$ be a random field that admits a $(a,b)$-sequence of finite range approximations with $a>0$ and $b>1$ (see Definition \ref{def:approximations}). For $\rho\geq 1$ let $\mathcal{R}^h_\rho := [0,\rho]\times [0,1]$ and $\mathcal{R}^v_\rho = [0,1]\times [0,\rho]$.
Let $\ell\in \R$. Assume that there exists $\rho_0>1$ such that the following holds
\begin{align}
    \inf_{x\in \R^2}\Proba{g+\ell\in \cross^h(x+\lambda\mathcal{R}^h_{\rho_0})}\xrightarrow[\lambda\to \infty]{}1, \\
    \inf_{x\in \R^2}\Proba{g+\ell\in \cross^v(x+\lambda\mathcal{R}^v_{\rho_0})}\xrightarrow[\lambda\to \infty]{}1.
\end{align}
Then, for any $\ell'>\ell$, almost surely, the set $\mathcal{E}_{\ell'}(g) = \{g+\ell'\geq 0\}$ contains a unique unbounded component and the set $\{g+\ell'<0\}$ does not contain any unbounded connected component.
\end{theorem}
Since the ideas of the proof of Theorem \ref{thm:crossings} are pretty classical, we defer the proof of Theorem \ref{thm:crossings} to Appendix \ref{sec:appendixA}.
In order to prove Theorem \ref{thm:continuous_l_grand} we also need the following intermediate result.
\begin{lemma}
\label{lemma:seed_high_prob}
Let $f=q\ast W$ where $q$ satisfies Assumption \ref{a:q} for $\beta>1.$ There exists a sequence $(\ell_\lambda)_{\lambda\in \N^*}$ such that
\begin{equation}
    \Proba{\forall z \in [0,\lambda]^2,\ \alpha_{\lambda}^{f_\lambda}(z)\leq \ell_\lambda} \xrightarrow[\lambda \to \infty]{}1.
\end{equation}
\end{lemma}
\begin{proof}
Let $\lambda \in \N^*$. By Lemma \ref{lemma:technical} we have
\begin{equation}
\label{eq:intermediate1789}
    \Proba{\forall z\in [0,1]^2,\ \alpha^f(z) \leq \ell}\xrightarrow[\ell \to \infty]{}1.
\end{equation}
Indeed, the fifth item of Lemma \ref{lemma:technical} show that almost surely there exists a finite $C>0$ such that
$$\forall z \in [0,1]^2,\ \alpha^f(z)= \sup_{r\in [0,C]}\tau_f(z,z+re_1).$$
By continuity and compactness this implies \eqref{eq:intermediate1789}. By \eqref{eq:intermediate1789} we may find $\ell_\lambda\in \R$ big enough so that
\begin{equation}
    \label{eq:intermdiate1790}
    \Proba{\exists z\in [0,1]^2,\ \alpha^f(z) > \ell_\lambda-1}\leq \frac{1}{\lambda^3}.
\end{equation}
Therefore, by stationarity and an union bound, we find using \eqref{eq:intermdiate1790} that
\begin{equation}
    \label{eq:intermediate1791}
    \Proba{\forall z\in [0,\lambda]^2,\ \alpha^f(z)\leq \ell_\lambda-1}\geq 1-\frac{1}{\lambda}\xrightarrow[\lambda\to \infty]{}1.
\end{equation}
Moreover, by Corollary \ref{cor:trunc_3} together with an union bound we find that
\begin{equation}
    \label{eq:intermediate1792}
    \Proba{\forall z\in [0,\lambda]^2,\ \left|\alpha_\lambda^{f_\lambda}(z) - \alpha^f(z)\right|<1}\geq 1 - C\lambda^2 e^{-c\lambda^b},
\end{equation}
where $0<b<\min\left(\frac{2\beta-2}{3},2\right)$ and where $c,C>0$ are two positive constants depending on $q$ and $b$.
Finally, \eqref{eq:intermediate1792} together with \eqref{eq:intermediate1791} conclude the proof.
\end{proof}
We now provide the proof of Theorem \ref{thm:continuous_l_grand}.
\begin{proof}[Proof of Theorem \ref{thm:continuous_l_grand}.]
We will apply our renormalization argument in dimension $d=2$. Denote by $b_0$ the following quantity
\begin{equation}
    \label{eq:def_exponent_b}
    b_0:= \min\left(2,\frac{2\beta-2}{3}\right).
\end{equation}
Since $\beta>\frac{5}{2}$ we have $b_0>1$. In the following we fix $b\in ]1,b_0[$ (for instance $b=\frac{1+b_0}{2}$).
We introduce two sequences $\mu=(\mu_n)_{n\in \N}$ and $\sigma=(\sigma_n)_{n\in \N}$ which are actually stationary sequences.
$$\forall n\in \N,\ \mu_n := 10^6 \text{ and } \sigma_n := 100.$$
It follows that $\mu$ and $\sigma$ satisfy Conditions \ref{cond:c1} and \ref{cond:c3}. Let $\varepsilon_0>0$ given by Proposition \ref{prop:renorm_high_prob}. Let $(\ell_\lambda)_{\lambda\in \N^*}$ be the sequence given by Lemma \ref{lemma:seed_high_prob}.
Let $\lambda_0\in \N^*$ (this parameter will be fixed later). Depending on $\lambda_0$ we have a sequence $(\lambda_n)_{n\in \N}$ which we recall is given by
$$\forall n\in \N,\ \lambda_{n+1}:=\mu_n\lambda_n.$$
We also introduce $\mathcal{R}_{\lambda_0}$ a renormalization scheme associated to $(\lambda_0,\mu,\sigma).$ In order to define this renormalization scheme we define its collection of seed events and auxiliary events. For the seed events, for $u\in \mathbb{L}_0=\lambda_0\mathbb{Z}^2$ we set
\begin{equation}
    \label{eq:def_seed_event_1}
    \mathcal{A}^{(0)}_u := \left\{\forall z\in u+[0,\lambda_0]^2,\ \alpha^{f_{\lambda_0}}_{\lambda_0}(z)\leq \ell_{\lambda_0}\right\}.
\end{equation}
For the auxiliary events, given $n\in \N^*$ and $u\in \mathbb{L}_n = \lambda_n\Z^2$ we set
\begin{equation}
    \mathcal{B}_u^{(n)} := \left\{\norm{\alpha_{\lambda_n}^{f_{\lambda_n}} - \alpha_{\lambda_{n-1}}^{f_{\lambda_{n-1}}}}{u+[0,\lambda_n]^2} \leq \frac{1}{2^n} \right\}.
\end{equation}
We see from the definition of the auxiliary and seed events and the definition of $\alpha^{f_\lambda}_\lambda$ that Condition \ref{cond:c2} is verified (for any choice for $\lambda_0$). Indeed, the seed event $\mathcal{A}_u^{(0)}$ is measurable with respect to the restriction of the white noise $W$ to the box $u+[-4\lambda_0,4\lambda_0]^2$ which is included in the formal support of $\mathcal{A}_u^{(0)}$ since $\sigma_0= 100$ (see Definition \ref{def:support}). Similarly, the auxiliary event $\mathcal{B}_u^{(n)}$ is measurable with respect to the restriction of the white noise $W$ in the box $u+[-4\lambda_n,4\lambda_n]^2$ which again is included in the formal support of $\mathcal{B}_u^{(n)}$ by our choice of $\sigma.$

We now show that when $\lambda_0$ is big enough, then Condition \ref{cond:c4} is verified for $\varepsilon_0.$
 In fact, applying Corollary \ref{cor:trunc_3} together with a triangular inequality we get for any $n\in \N^*$ and for any $u\in \mathbb{L}_n$,
\begin{equation}
    \Proba{\mathcal{B}_u^{(n)}} \geq 1-C\lambda_n^2e^{-c(\lambda_{n-1})^{b}\frac{1}{4^n}}.
\end{equation}
Since $\lambda_n = \mu_{n-1} \lambda_{n-1}=10^6\lambda_{n-1}$ and since $(\lambda_{n-1})^{b}\geq (\lambda_0)^b(10^6)^{n-1}$ (which uses $b>1$) we see that adjusting constants we may find $\lambda_0$ big enough such that for all $n\in \N^*$ and $u\in \mathbb{L}_n$,
\begin{equation}
    \label{eq:intermdiate4798}
    \Proba{\mathcal{B}_u^{(n)}}\geq 1-2^{-c\lambda_02^n}.
\end{equation}
Lemma \ref{lemma:seed_high_prob} together with \eqref{eq:intermdiate4798} show that when $\lambda_0$ is big enough then $\mathcal{R}_{\lambda_0}$ satisfies Condition \ref{cond:c4} for $\varepsilon_0$. We now fix $\lambda_0\in \N^*$ big enough such that Condition \ref{cond:c4} holds, and we work with the renormalization scheme $\mathcal{R}_{\lambda_0}$. In particular, 
Proposition \ref{prop:renorm_high_prob} yields
\begin{equation}
    \label{eq:high_prob_2}
    \forall n\in \N,\ \forall u\in \mathbb{L}_n,\  \Proba{\mathcal{A}^{(n)}_u}\geq 1-2^{-2^n}.
\end{equation}
Let $\lambda>\lambda_0$ and $N=N(\lambda)\in \N$ be the largest integer $N\in \N$ such that $\lambda_N <\lambda.$ Denote by $k_0=k_0(\lambda)$ the integer $k_0 := \lceil 2\lambda/\lambda_N\rceil.$
For $-1\leq k\leq k_0$ we denote by $u_k\in \mathbb{L}_N$ the vertex
\begin{equation}
    u_k := (k\lambda_N, 0)\in \mathbb{L}_N.
\end{equation}
Note that since $10^6 \lambda_N\geq \lambda$ we have $k_0\leq 2\times 10^6+2$ (which is an absolute constant).
Denote by $\mathcal{G}_\lambda$ the following event,
\begin{equation}
    \mathcal{G}_\lambda:= \bigcap_{k=-1}^{k_0}\mathcal{A}^{(N)}_{u_k}.
\end{equation}
Due to \eqref{eq:high_prob_2} and since $N\geq \frac{\log(\lambda/\lambda_0)}{\log(10^6)}-1$ we find by an union bound that
\begin{equation}
    \label{eq:high_prob_G_lambda}
    \Proba{\mathcal{G}_\lambda}\xrightarrow[\lambda \to \infty]{}1.
\end{equation}
On the event $\mathcal{G}_\lambda$, the vertices $u_{-1},\dots, u_{k_0}$ form in $\mathbb{L}_{N}$ a nearest-neighbor path of points good up to scale $N$. Therefore applying repeatedly Lemma \ref{lemma:der_renorm_extract_path}, we may find an integer $l\in \N^*$ and vertices $v_1,\dots,v_l$ of $\mathbb{L}_0$ that form a nearest-neighbor path of points in $\mathbb{L}_0$ such that $v_1 \in [-\infty,0]\times [0,\lambda_N]$, $v_l\in [2\lambda, \infty[\times [0,\lambda_N]$ and such that all points $v_i$ are good up to scale $N$.
In particular, for any $1\leq i \leq l$, in the box $B_i := v_i+[0,\lambda_0]^2$ the following holds:
\begin{itemize}
    \item $\forall z\in B_i,\ \alpha^{f_{\lambda_0}}_{\lambda_0}(z)\leq \ell_{\lambda_0}$
    \item $\forall z\in B_i,\ \forall n\in \llbracket 1, N\rrbracket,\ |\alpha^{f_{\lambda_n}}_{\lambda_n}(z) - \alpha^{f_{\lambda_{n-1}}}_{\lambda_{n-1}}(z)|\leq \frac{1}{2^n}.$
\end{itemize}
In fact, the first item is due to the fact that the event $\mathcal{A}^{(0)}_{v_i}$ must occur, and the second item comes from the fact that $v_i$ sits in a tower of scales where all the events $\mathcal{B}^{(n)}_u$ occur (for $n\leq 1\leq N$ and $u\in \mathbb{L}_n$ such that $u+[0,\lambda_n[^2$ contains $v_i$).
In particular, for all $1\leq i \leq l$ and for all $z\in B_i$ then,
\begin{equation}
    \alpha^{f_{\lambda_N}}_{\lambda_N}\leq \ell_{\lambda_0}+1.
\end{equation}
Denote by $\mathcal{G}'_\lambda$ the event,
\begin{equation}
    \mathcal{G}'_\lambda := \left\{\norm{\alpha^f-\alpha^{f_{\lambda_N}}_{\lambda_N}}{[0,2\lambda]\times[0,\lambda]}<1\right\}.
\end{equation}
Then, on the event $\mathcal{G}_\lambda\cap \mathcal{G}'_\lambda$, for all $1\leq i \leq l$ and for all $z\in B_i$ we have
\begin{equation}
    \alpha^f(z)\leq \ell_{\lambda_0}+2.
\end{equation}
Note that the union of all boxes $B_i$ is a connected set that is included in $\mathbb{R}\times [0,\lambda]$ and that intersects $\{0\}\times [0,\lambda]$ and $\{2\lambda\}\times [0,\lambda].$ Hence, on the event $\mathcal{G}_\lambda \cap \mathcal{G}'_\lambda$ there exists a horizontal crossing of the rectangle $[0,2\lambda]\times [0,\lambda]$ by the set $\{\alpha^f\leq \ell_{\lambda_0}+2\}.$
It remains to see that the event $\mathcal{G}_\lambda \cap \mathcal{G}'_\lambda$ has high probability. We already know by  \eqref{eq:high_prob_G_lambda} that $\mathcal{G}_\lambda$ has high probability. Moreover, by Corollary \ref{cor:trunc_3} (and the fact that $\lambda_N \geq \lambda/10^6$) we have
\begin{equation}
    \label{eq:high_prob_G_lambda_2}
    \Proba{\mathcal{G}'_\lambda}\geq 1-C\lambda_N^2e^{-c\lambda_N^b}.
\end{equation}
This also shows that $\mathcal{G}'_\lambda$ has high probability. By an union bound, the intersection $\mathcal{G}_\lambda\cap \mathcal{G}'_\lambda$ also has high probability. Therefore, we have proven that the set $\{\alpha^f\leq \ell_{\lambda_0}+2\}$ crosses horizontal rectangles of aspect ratio two by one with high probability. The exact same argument also applies to see that $\{\alpha^f \leq \ell_{\lambda_0}+2\}$ crosses vertical rectangles with high probability. We can now apply Theorem \ref{thm:crossings} (the field $\alpha^f$ admits a $(a,b)$-sequence of finite range approximations with $a>0$ and $b>1$ by Corollary \ref{cor:trunc_3}). This yields that for any $\ell>\ell_{R_0}+2$, almost surely, the set $\{\alpha^f\leq \ell\}$ contains a unique unbounded component and the set $\{\alpha^f \geq \ell\}$ does not contain any unbounded component. This concludes the proof of Theorem \ref{thm:continuous_l_grand}.
\end{proof}

\section{Proof of Theorem \ref{thm:discrete_correlated}}
\label{sec:5}
This section is dedicated to the proof of Theorem \ref{thm:discrete_correlated}. Recall the definitions of $X$ and $\alpha_X$ in \eqref{eq:defX} and \eqref{def:alpha_intro_discrete}.
Similarly to what was done in the continuous setting for $R,R_0\geq 1$ we introduce the following fields.
\begin{align}
    &\alpha^X_R(u) := \sup_{r\in \llbracket 1,R\rrbracket}\frac{f(u+re_1)-f(u)}{r} \\
   & \alpha^{X_{R_0}}(u) := \sup_{r\in \mathbb{N}^*}\frac{f_{R_0}(u+re_1)-f_{R_0}(u)}{r}\\
   & \alpha^{X_{R_0}}_R(u) := \sup_{r\in \llbracket 1,R\rrbracket}\frac{f_{R_0}(u+re_1)-f_{R_0}(u)}{r}.
\end{align}
and we observe that Lemmas \ref{lemma:trunc_1}, \ref{lemma:trunc_2} and Corollary \ref{cor:trunc_3} also hold for these approximations of $\alpha^X$ (by using the same arguments). One would think that we could directly apply the same proof as for Theorem \ref{thm:continuous_l_grand} to obtain Theorem \ref{thm:discrete_correlated}. It appears that the same proof is indeed applicable for the regime of large $\ell.$ However, for the regime of small $\ell$ then due to a sprinkling and to the fact that one does not quantify the sequence $(\ell_\lambda)$ in Lemma \ref{lemma:seed_high_prob}, then we could end-up proving that $\{\alpha^X\geq \ell\}$ admits a unique infinite cluster for $\ell<0$ which is completely trivial and not what we would like to prove. We need to be more careful about our approach. A key result of this section if Proposition \ref{prop:peierls_correl}.
\begin{proposition}
\label{prop:peierls_correl}
Let $f=q\ast W$ where $q$ satisfies Assumption \ref{a:q} for some $\beta>1$ and Assumption \ref{a:q2} for $(1,\delta)$, where $\delta>0$.
Let $\rho\in ]0,1[$. There exists $\delta_0(\rho)>0$ (depending only on $\rho$) such that, if $\delta<\delta_0(\rho)$, there exists a constant $R_0\in \N^*$ such that for any $A\subset \Z^2$ finite subset and any $R\geq R_0$ then,
\begin{equation}
    \Proba{\forall u \in A,\ \alpha^X_R(u)\leq 0} \leq \rho^{|A|}.
\end{equation}
\end{proposition}
We first show how Proposition \ref{prop:peierls_correl} allows us to conclude. We use Proposition \ref{prop:peierls_correl} together with a Peierls argument to prove the following technical result.
\begin{lemma}
\label{lemma:seed_high_prob2}
Let $f=q\ast W$ where $q$ satisfies Assumption \ref{a:q} for some $\beta>16$ and \ref{a:q2} for $(1,\delta).$
Given $\lambda \in \N^*$ we denote by $\mathcal{A}_\lambda$ the event that there exists a nearest-neighbor loop $\gamma$ in $\Z^2$ that separates $\partial [0,\lambda]^2$ from $\partial [-\lambda,2\lambda]^2$ and that only uses vertices $v\in \Z^2$ such that $\alpha^{X_\lambda}_\lambda(v) > \frac{1}{\lambda^5}.$ There exists an absolute constant $\delta_0>0$ such that if $\delta\in ]0,\delta_0[$ then,
\begin{equation}
    \label{eq:Au0goesto11}
    \Proba{\mathcal{A}_\lambda}\xrightarrow[\lambda\to \infty]{}1.
\end{equation}
\end{lemma}
\begin{proof}
We write $\mathcal{G}_\lambda^1\cap \mathcal{G}_\lambda^2\cap \mathcal{G}_\lambda^3\subset \mathcal{A}_\lambda$ where
\begin{itemize}
    \item $\mathcal{G}_\lambda^1$ is the event that there exists a nearest-neighbor loop in $\Z^2$ separating $\partial [0,\lambda]^2$ from $\partial[-\lambda,2\lambda]^2$ that only uses vertices $v\in \Z^2$ such that $\alpha^{X}_{\lambda}(v) > 0$
    \item $\mathcal{G}_\lambda$ is the event that there does not exist a vertex $v\in [-\lambda,2\lambda]^2\cap \mathbb{Z}^2$ such that $\alpha^X_{\lambda}(v)\in [0,\frac{2}{(\lambda_0)^5}].$
    \item $\mathcal{G}_3$ is the event that for all $v\in [-\lambda,2\lambda]^2\cap \Z^2$ we have $$|\alpha^{X}_{\lambda}(v)- \alpha^{X_{\lambda}}_{\lambda}(v)|<\frac{1}{\lambda^5}.$$
\end{itemize}
In order to prove \eqref{eq:Au0goesto11} it is enough to prove that the three events $\mathcal{G}_\lambda^1$, $\mathcal{G}_\lambda^2$ and $\mathcal{G}_\lambda^3$ have high probability when $\lambda$ goes to infinity. For $\mathcal{G}_\lambda^1$, take $\rho=\frac{1}{10}$ and $\delta_0>0$ given by Proposition \ref{prop:peierls_correl}. Then by a classical Peierls argument we see that $\mathcal{G}_\lambda^1$ has high probability when $\lambda$ goes to infinity. For $\mathcal{G}_\lambda^2$, by stationarity and a union bound we find that
\begin{align*}
    \Proba{(\mathcal{G}_\lambda^2)^c}&\leq 10\lambda^2\Proba{\alpha^X_{\lambda}(0)\in [0,2\lambda^{-5}]} \\
    & \leq 10\lambda^2\Proba{\exists r\in \llbracket 1, \lambda\rrbracket,\  0<X_{r,0}-X_{0,0}<2\lambda^{-4}} \\
    & \leq O\left(\lambda^{-1}\right).
\end{align*}
where in the last line we use the fact that each $X_{r,0}-X_0$ is a Gaussian random variable which is centered and of variance lower bounded by some positive quantity independent from $r$. This proves that the event $\mathcal{G}_\lambda^2$ has high probability when $\lambda$ goes to $\infty$. Finally, by Lemma \ref{lemma:trunc_1} (applied for $\varepsilon=\lambda^{-5}$ and $R=R_0=\lambda$) and an union bound, the event $\mathcal{G}_3$ has high probability as soon as 
$$\frac{2\beta-2}{3}> 10.$$
This condition is satisfied since we assume $\beta>16.$ This concludes the proof of Lemma \ref{lemma:seed_high_prob2}.
\end{proof}
We now provide the proof of Theorem \ref{thm:discrete_correlated}.
\begin{proof}[Proof of Theorem \ref{thm:discrete_correlated}]
We only prove the fact that $\{\alpha>\ell_2\}$ contains a unique unbounded cluster when $\ell_2>0$ is small enough. In fact the same arguments as in the proof of Theorem \ref{thm:continuous_l_grand} easily imply that $\{\alpha<\ell_1\}$ contains a unique unbounded cluster when $\ell_1$ is big enough.

Let $\mu=(\mu_n)_{n\in \N}$ and $\sigma = (\sigma_n)_{n\in \N}$ be two sequences of integers that are stationary and defined by
$$\forall n\in \N,\ \mu_n:=10^6\text{ and } \sigma_n := 100.$$
Then $\mu$ and $\sigma$ verify Conditions \ref{cond:c1} and \ref{cond:c3}. We therefore denote by $\varepsilon_0>0$ the parameter given by Proposition \ref{prop:renorm_high_prob}.
Let $\lambda_0\in \mathbb{N}^*$ to be fixed later. Recall that we define a sequence $(\lambda_n)_{n\in \N}$ by
$$\forall n\in \N,\ \lambda_{n+1}:= \mu_n\lambda_n.$$
Depending on $\lambda_0$ we introduce a renormalization scheme $\mathcal{R}_{\lambda_0}$ associated to $(\lambda_0,\mu,\sigma)$ for which the seed events and auxiliary events are defined as follows. For the seed events, if $u\in \mathbb{L}_0=\lambda_0\Z^2$, we denote by $\mathcal{A}_u^{(0)}$ the event that there exists a nearest-neighbor loop in $\Z^2$ separating $\partial (u+[0,\lambda_0]^2)$ and $\partial(u+[-\lambda_0,2\lambda_0]^2)$ and included in $\{\alpha^{X_{\lambda_0}}_{\lambda_0} > \frac{1}{(\lambda_0)^{5}}\}.$
For the auxiliary events, for $n\in \N^*$ and $u\in \mathbb{L}_n=\lambda_n\mathbb{Z}^2$, we denote by $\mathcal{B}_u^{(n)}$ the following event
\begin{equation}
    \mathcal{B}_u^{(n)}:= \left\{\forall v\in u+[-2\lambda_n,2\lambda_n]^2\cap \Z^2,\ |\alpha_{\lambda_0}^{X_{\lambda_n}}(v)-\alpha_{\lambda_0}^{X_{\lambda_{n-1}}}(v)|\leq \frac{2^{-n}}{2\lambda_0^5}\right\}.
\end{equation}
By Lemma \ref{lemma:trunc_1}, and since $\frac{2\beta-2}{3}>10$ we have
\begin{equation}
    \Proba{\left(\mathcal{B}_u^{(n)}\right)^c} \leq C\lambda_0^3(10^6)^{2n}\exp\left(-c4^{-n}(\lambda_0)^{\frac{2\beta-2}{3}-10}(10^6)^n\right).
\end{equation}
Therefore, we may find a small $\eta>0$ such that when $\lambda_0\geq 1$ is big enough, we have
\begin{equation}
    \label{eq:intermediate479863}
    \forall n\in \N^*,\ \Proba{\left(\mathcal{B}_u^{(n)}\right)^c} \leq 2^{-c\lambda_0^\eta 2^n}
\end{equation}
The rest of the proof is similar to the proof of Theorem \ref{thm:continuous_l_grand} and we describe the steps shortly. By Lemma \ref{lemma:seed_high_prob2} and \eqref{eq:intermediate479863} we see that when $\lambda_0$ is big enough then $\mathcal{R}_{\lambda_0}$ verifies Condition \ref{cond:c4} for $\varepsilon_0$. Moreover, as in the proof of Theorem \ref{thm:continuous_l_grand}, it is easy to check that $\mathcal{R}_{\lambda_0}$ also satisfies Condition \ref{cond:c2} (for any $\lambda_0\in \N^*$). We now fix $\lambda_0$ big enough so that Condition \ref{cond:c4} is verified. Proposition \ref{prop:renorm_high_prob} implies that,
\begin{equation}
    \forall n\in \N^*,\ \forall u\in \mathbb{L}_n,\ \Proba{\mathcal{A}^{(n)}_u}\geq 1-2^{-2^n}.
\end{equation}
Let $\lambda> \lambda_0$ and $N=N(\lambda)\in \N$ be the largest integer such that $L_N\leq \lambda$.
By the same construction as in the proof of Theorem \ref{thm:continuous_l_grand}. We see that with high probability (in terms of $\lambda$) we may find an nearest-neighbor path $v_0,\dots,v_k$ in $\mathbb{L}_0$ such that
\begin{enumerate}
    \item All $v_i$ are in $[-\lambda,3\lambda]\times [0,\lambda]$
    \item $v_0$ is in $[-\lambda,0]\times [0,\lambda]$ and $v_k$ in $[2\lambda,3\lambda]\times [0,\lambda].$
    \item For any $0\leq i\leq k$, the event $\mathcal{A}^{(0)}_{v_i}$ occurs.
    \item For any $0\leq i \leq k$, for any $1\leq n\leq N$, then for all $w\in v_i+[-\lambda_0,2\lambda_0]^2\cap \Z^2$ we have $$\left|\alpha_{\lambda_0}^{X_{\lambda_n}}(w)-\alpha_{\lambda_0}^{X_{\lambda_{n-1}}}(w)\right|\leq \frac{2^{-n}}{2\lambda_0^5}.$$
\end{enumerate}
This implies that with high probability then the rectangle $[0,2\lambda]\times [0,\lambda]$ is crossed by a nearest-neighbor  path of $\Z^2$ included in $\{\alpha^{X_{\lambda_N}}_{\lambda_0}\geq \frac{1}{2\lambda_0^5}\}.$
By Lemma \ref{lemma:trunc_1}, with high probability we have $|\alpha^{X_{\lambda_N}}_{\lambda_0}-\alpha^X_{\lambda_0}|\leq \frac{1}{4\lambda_0^5}$ on $[0,2\lambda]\times [0,\lambda]^2$. Therefore, with high probability the rectangle $[0,2\lambda]\times [0,\lambda]$ is crossed by an a nearest-neighbor path of $\Z^2$ included in $\{\alpha^X_{\lambda_0}\geq \frac{1}{4\lambda_0^5}\}.$
Since $\alpha^X\geq \alpha^X_{\lambda_0}$ we have proven that with high probability the rectangle
$[0,2\lambda]\times [0,\lambda]$ is crossed by an orthogonally connected path in $\{\alpha^X \geq \frac{1}{4\lambda_0^5}\}.$ The same is true for vertical rectangles and we may apply Theorem \ref{thm:crossings} to conclude the proof of Theorem \ref{thm:discrete_correlated}.
\end{proof}
The rest of this section is dedicated to the proof of Proposition \ref{prop:peierls_correl}.
We start by introducing some notations.
When $0<\delta<1$ we write
\begin{equation}
    \label{eq:def_CR}
    C(\delta) := \sqrt{\frac{1+\delta}{1-\delta}}
\end{equation}
We also introduce the following notations.
Let $A\subset \Z^2$ be a finite subset, we write $A=\{a_1,\dots,a_n\}$ where $n=|A|$ and where the $a_i$ are ordered according to the lexicographic order of $\Z^2$. When $x=(x_{a_1},\dots,x_{a_n}) \in \R^A$ and $\sigma \in \mathfrak{S}_n$ is a permutation over $\{1,\dots,n\}$ we say that $x$ follows the order of $\sigma$ and we write $x\sim \sigma$ if
\begin{equation}
    x_{a_{\sigma^{-1}(1)}}>x_{a_{\sigma^{-1}(2)}}>\dots>x_{a_{\sigma^{-1}(n)}}.
\end{equation}
We also denote by $\sigma\cdot x$ the element of $\R^A$ defined as
$$\sigma\cdot x = (x_{a_{\sigma(1)}}, \dots, x_{a_{\sigma(n)}}).$$
And we observe that if $x\sim \sigma$ then $\sigma'\cdot x\sim \sigma'\circ \sigma$.
When $x\in \R^A$ and $B\subset A$ is a subset of $A$ we write $x_{|B}\in \R^B$ the restriction of $x$ by keeping only the coordinates in $B$.

\begin{lemma}
\label{lemma:order}
Let $A\subset \Z^2$ be a finite subset. Let $R := \min\{\norm{u-v}{\infty} | u\neq v\in A\}$.
Let $f=q\ast W$ where $q$ satisfies Assumption \ref{a:q} for some $\beta>1$ and Assumption \ref{a:q2} for $(R,\delta)$ with some $0<\delta<1$.
Let $A_1,\dots,A_m$ be a partition of $A$. For $1\leq i \leq m$, denote by $n_i$ the cardinal of $A_i$ and let $\sigma_i \in \mathfrak{S}_{n_i}$. Then,
\begin{equation}
    \label{eq:order}
    \prod_{i=1}^m\frac{\left(C(\delta)^{-1}\right)^{n_i}}{n_i!}\leq \Proba{\forall 1\leq i\leq m,\ f_{|A_i}\sim \sigma_i} \leq \prod_{i=1}^m\frac{\left(C(\delta)\right)^{n_i}}{n_i!},
\end{equation}
where we recall that $C(\delta)=\sqrt{\frac{1+\delta}{1-\delta}}.$
\end{lemma}
\begin{proof}
Denote by $X_A := f_{|A} \in \R^A$. Then $X_A$ is a non degenerated Gaussian vector. We write $\Sigma_A$ the covariance matrix of $X_A$ which is a positive definite matrix.
Let $n:=|A|$. Let $\sigma=(\sigma_1,\dots,\sigma_m) \in \mathfrak{S}_{n_1}\times\cdots \times \mathfrak{S}_{n_m}$. We see that $\sigma$ naturally corresponds to a permutation of $\mathfrak{S}_n$. We denote by $F_A^{\sigma}$ the subset of $\R^A$ defined by
\begin{equation}
    F_A^{\sigma} := \{x\in \R^A\ |\ \forall 1\leq i \leq m,\ x_{|A_i}\sim \sigma_i\}.
\end{equation}
We can rewrite the probability in \eqref{eq:order} as
\begin{equation}
    \label{eq:order_2}
    \Proba{X_A\in F_A^{\sigma}}=\frac{1}{Z_A}\int_{F_A^{\sigma}}\exp\left(-\frac{1}{2}\scal{x}{(\Sigma_A)^{-1}x}\right)dx,
\end{equation}
where $dx$ denotes the Lebesgue measure on $\R^A$, $Z_A := \sqrt{(2\pi)^n\det(\Sigma_A)}>0$ and $\scal{x}{y}$ denotes the usual Euclidean scalar product between two elements of $\R^A$ (that is, with matrix notations, $\scal{x}{y}=x^Ty$).

Let $\sigma'=(\sigma'_1\dots, \sigma'_m)$ be another element of $\mathfrak{S}_{n_1}\times \cdots\times \mathfrak{S}_{n_m}$.
Then, we observe that the sets $F_A^{\sigma}$ and $F_A^{\sigma'}$ are related by the following relation
$$x\in F_A^{\sigma'}\Leftrightarrow P^{\sigma'\circ\sigma^{-1}}x \in F_A^{\sigma},$$
where $P^{\sigma'\circ\sigma^{-1}}$ is the permutation matrix associated to the permutation $\sigma'\circ \sigma^{-1}$ (that is, $P^{\sigma'\circ\sigma^{-1}}$ is a matrix diagonal by block, and on each $\R^{A_k}$ it is given by the permutation matrix associated to $\sigma'_k \circ (\sigma_k)^{-1}$).
Doing the change of coordinate $x=P^{\sigma'\circ \sigma^{-1}}y$ in \eqref{eq:order_2} and since $P^{\sigma'\circ \sigma^{-1}}$ is an orthogonal matrix  we obtain
\begin{equation}
    \label{eq:order_3}
    \Proba{X_A\in F_A^{\sigma}}=\frac{1}{Z_A}\int_{F_A^{\sigma'}}\exp\left(-\frac{1}{2}\scal{P^{\sigma'\circ\sigma^{-1}}y}{(\Sigma_A)^{-1}P^{\sigma'\circ\sigma^{-1}}y}\right)dy.
\end{equation}
Now given $P$ an orthogonal matrix on $\R^A$ we see that $(\Sigma_A)^{-1}$ and $P^T (\Sigma_A)^{-1}P$ are two positive definite matrices with the same eigenvalues. Denote by $\lambda_{\text{max}}>0$ and $\lambda_{\text{min}}>0$ the maximal and minimal eigenvalues of $\Sigma_A$.
Then for any $x\in \R^A$ we have
\begin{equation}
    \frac{1}{\lambda_\text{max}}\norm{x}{}^2\leq x^TP^T\Sigma_A^{-1}Px\leq \frac{1}{\lambda_\text{min}}\norm{x}{}^2.
\end{equation}
And therefore,
\begin{equation}
    \frac{\lambda_\text{min}}{\lambda_\text{max}}x^T\Sigma_A^{-1}x \leq x^TP^T\Sigma_A^{-1}Px\leq \frac{\lambda_\text{max}}{\lambda_\text{min}}x^T\Sigma_A^{-1}x.
\end{equation}
Let $C_A=\frac{\lambda_\text{max}}{\lambda_\text{min}}\geq 1$.
By using these inequalities in \eqref{eq:order_3} and doing a change of variable $y=\sqrt{C_A}z$ we get
\begin{align*}
    \Proba{f_{|A}\in F_A^{\sigma}} & \leq \frac{1}{Z_A}\int_{F_A^{\sigma'}}\exp\left(-\frac{1}{2C_A}\scal{y}{(\Sigma_A)^{-1}y}\right)dy. \\
    &\leq \frac{C_A^{n/2}}{Z_A}\int_{F_A^{\sigma'}}\exp\left(-\frac{1}{2}\scal{z}{(\Sigma_A)^{-1}z}\right)dy\\
    &\leq C_A^{n/2}\Proba{f_{|A}\in F_A^{\sigma'}}.
\end{align*}
By the same argument, a similar lower bound is obtained, and we see that for all $\sigma,\sigma'$ we have
\begin{equation}
    (C_A)^{-n/2}\Proba{X_A\in F_A^{\sigma'}} \leq \Proba{X_A\in F_A^\sigma}\leq (C_A)^{n/2}\Proba{X_A\in F_A^{\sigma'}}.
\end{equation}
We can sum these inequalities over all $\sigma'$. Note that there are exactly $n_1!\dots n_m!$ choices for $\sigma'$ and that since $f_{|A}$ is a non degenerated Gaussian vector we have
$$\sum_{\sigma'\in \mathfrak{S}_{n_1}\times \cdots\times \mathfrak{S}_{n_m}}\Proba{f_{|A}\in F_A^{\sigma'}}=1.$$
We obtain
\begin{equation}
    (C_A)^{-n/2}\leq \left(\prod_{i=1}^m n_i!\right)\Proba{f_{|A}\in F_A^\sigma}\leq (C_A)^{n/2}.
\end{equation}
Therefore to conclude to the proof of the lemma, it is enough to show that one has $C_A \leq C(\delta)^2= \frac{1+\delta}{1-\delta}.$
This is a direct application of Gershgorin circle theorem. In fact the diagonal entries of $\Sigma_A$ are equal to $1$ and along each row, the sum of the absolute values of the non diagonal entries is upper bounded by $\delta$. Therefore all eigenvalues of $\Sigma_A$ lie in the interval $[1-\delta,1+\delta].$
This concludes the proof of this lemma.
\end{proof}
We conclude this section with the proof of Proposition \ref{prop:peierls_correl}.
\begin{proof}[Proof of Proposition \ref{prop:peierls_correl}]
Let $\rho\in ]0,1[$. Let $\delta_0\in ]0,1/2[$ to be fixed later (depending on $\rho$). Let $f=q\ast W$ be such that Assumption \ref{a:q2} is verified for $(1,\delta)$ with $\delta\in ]0,\delta_0[$ Let $R\in \N^*$ be a free parameter for now.
Let $I:= \{i\in \Z\ |\ A\cap \Z\times \{i\}\neq \emptyset\}.$
We write $A=\sqcup_{i\in I}A^i\times \{i\}$ where each $A^i$ is a subset of $\Z.$
For each $A^i$ we write it as,
$$A^i = \sqcup_{j=1}^{n_i}B^i_j.$$
where the $B^i_j$ are non-empty subsets of $A^i$ such that
\begin{itemize}
    \item Each set $B^i_j$ is \textit{$R-$connected}, that is, for any $b\in B^i_j$, either $b$ is the maximum of $B^i_j$ or there exists $b'\in B^i_j$ such that $b<b'\leq b+R.$
    \item For any $1\leq j\leq n_i-1$ we have $\max B^i_j+R < \min B^i_{j+1}.$
\end{itemize}
Let $n_0=n_0(\rho)\geq 1$ be such that
\begin{equation}
    \forall n\geq n_0,\ \frac{2^n}{n!}\leq \rho^{2n}.
\end{equation}
We observe that my monotonicity of $\delta\mapsto C(\delta)=\sqrt{\frac{1+\delta}{1-\delta}}.$
Then for $0<\delta<\delta_0<\frac{1}{2}$ we have
\begin{equation}
    \label{eq:condition_n0}
    \forall n\geq n_0,\ \frac{C(\delta)^n}{n!}\leq \frac{C(\delta_0)^n}{n!} \leq \frac{2^n}{n!} \leq \rho^{2n}.
\end{equation}
Denote by $I(n_0)$ and $\tilde{I}(n_0)$ the sets
\begin{align}
    I(n_0):= \{(i,j)\in \Z\times \mathbb{N}^*\ |\ i\in I,\ 1\leq j\leq n_i,\ |B_j^i|\geq n_0\} \\
    \tilde{I}(n_0) := \{(i,j)\Z\times \mathbb{N}^*\ |\ i\in I,\ 1\leq j\leq n_i,\ |B_j^i|< n_0\}
\end{align}
We also denote by $S(n_0)$ and $\tilde{S}(n_0)$ the quantities
\begin{align}
    S(n_0) := \sum_{(i,j)\in I(n_0)}|B^i_j| \\
    \tilde{S}(n_0) := \sum_{(i,j) \in \tilde{I}(n_0)}|B^i_j|
\end{align}
Denote by $\mathcal{E}$ the event
\begin{equation}
    \mathcal{E}:= \left\{\forall u\in A,\ \alpha^X_R(u)<0\right\}.
\end{equation}
Then in particular, due to the definition of $\alpha^X_R$ and the fact that the $B_j^i$ are $R$-connected, the event $\mathcal{E}$ implies that for all $(i,j)$ we have $X_{|B_j^i}\sim \text{id}_{|B_j^i|}.$ We distinguish two cases. Assume first that $S(n_0)\geq \frac{|A|}{2}$. Then,
\begin{equation}
    \Proba{\mathcal{E}}\leq \Proba{\forall (i,j)\in I(n_0),\ X_{|B_j^i}\sim \text{id}_{|B_j^i|}}
\end{equation}
By Lemma \ref{lemma:order}, we find that
\begin{equation}
    \Proba{\mathcal{E}}\leq \prod_{(i,j)\in I(n_0)}\frac{C(\delta)^{|B_j^i|}}{|B_j^i|!}.
\end{equation}
By our choice of $n_0$ (see \eqref{eq:condition_n0}) and the definition of $I(n_0)$ we find
\begin{equation}
    \Proba{\mathcal{E}}\leq (\rho^2)^{S(n_0)}\leq \rho^{|A|}.
\end{equation}
This concludes the proof in the case $S(n_0)\geq \frac{|A|}{2}.$
We now assume that $S(n_0)\leq \frac{|A|}{2}$. Since $S(n_0)+\tilde{S}(n_0)=|A|$, this implies $\tilde{S}(n_0)\geq \frac{|A|}{2}.$ Moreover we have $\tilde{S}(n_0)\leq n_0|\tilde{I}(n_0)|.$ Therefore we have
\begin{equation}
    |\tilde{I}(n_0)|\geq \frac{|A|}{2n_0}.
\end{equation}
Denote by $A'\subset A$ a subset built by choosing (in an arbitrary way) exactly one point in each of the $B^i_j\times \{i\}$ for $(i,j)\in \tilde{I}(n_0)$. Therefore we have $|A'|\geq \frac{|A|}{2n_0}.$ Let $R_0 = R_0(\rho)\in \N^*$ be such that
\begin{equation}
    \label{eq:choice_R0}
    \frac{1}{R_0}\leq \frac{\rho^{2n_0}}{2}.
\end{equation}
We assume in the following that $R\geq R_0.$ By construction we see that two points of $A'$ that are on the same row of $\Z^2$ must be at distance at least $R$ from one another (henceforth at distance at least $R_0$) however two points on different rows may be at distance $1$. 
For $(u,v)\in A'$ we denote by $\mathcal{R}(u,v)$ the set
\begin{equation}
    \mathcal{R}(u,v) := \{(u+k,v)\ |\ 0\leq k< R_0\}.
\end{equation}
Denote by $\mathcal{R}$ the union of all $\mathcal{R}(u,v)$
\begin{equation}
    \mathcal{R} := \bigcup_{(u,v)\in A'} \mathcal{R}(u,v).
\end{equation}
By construction, for $(u,v)\neq (u',v')$ two elements of $A'$ then $\mathcal{R}(u,v)\cap \mathcal{R}(u',v')=\emptyset$, therefore $\mathcal{R}$ contains exactly $R_0|A|$ points.
Moreover on the event $\mathcal{E}$ we see that for each $(u,v)\in A'$, then $X_{u,v}$ must be maximum in the collection $(X_{u',v'})_{(u',v')\in R(u,v)}$
Note that there are $R_0$ elements in $R(u,v)$ and among the $R_0!$ permutations of $\mathfrak{S}_{R_0}$ there are $(R_0-1)!$ that have $1$ as a fixed point. Therefore, applying Lemma \ref{lemma:order} and summing over all these permutations one gets
\begin{equation}
    \label{eq:buir}
    \Proba{\mathcal{E}}\leq \prod_{(u,v)\in A'}\frac{C(\delta)^{R_0}}{R_0} \leq \prod_{(u,v)\in A'}\frac{C(\delta_0)^{R_0}}{R_0}
\end{equation}
We now choose $\delta_0>0$ small enough depending on $\rho$ such that
\begin{equation}
    \frac{C(\delta_0)^{R_0}}{R_0} \leq \rho^{2n_0}.
\end{equation}
This is possible due to our choice of $R_0$ (see \eqref{eq:choice_R0}) and since $C(\delta)$ can be made arbitrarily close to $1$ when $\delta$ is close to $0$.
To conclude, we see that for our parameters $\delta,R_0$ depending on $\rho$, then for any $R\geq R_0$ and $\delta<\delta_0$,
\begin{equation}
    \Proba{\mathcal{E}}\leq \prod_{(u,v)\in A'}\frac{C(\delta_0)^{R_0}}{R_0} \leq \rho^{2n_0|A'|} \leq \rho^{|A|},
\end{equation}
where the last inequality comes from the fact that we have $|A'|\geq \frac{|A|}{2n_0}.$ This concludes the proof of Proposition
\ref{prop:peierls_correl}.
\end{proof}

\appendix

\section{Gluing constructions for planar percolation}
\label{sec:appendixA}
This appendix is dedicated to the proof of Theorem \ref{thm:crossings}. The ideas in this appendix are not new and often referred to as \textit{gluing constructions}. In particular our proof of Theorem \ref{thm:crossings} is heavily inspired by what is done in \cite{BG17}, \cite{MV20}. The main difference is that we state and prove the result without assuming invariance by rotation (although it does not change much of the proof).
In the following we let $g: \R^2\to R$ be a random function defined on some probability space. We simply require that for each $\ell\in \R$, the excursion set \begin{equation}
    \mathcal{E}_\ell(g) := \{x\in \R^2\ |\ g(x)+\ell\geq 0\}
\end{equation}
is measurable.
First, recall Definition \ref{def:cross} of $\cross^h(\mathcal{R})$ and $\cross^v(\mathcal{R})$. We make the following definition.
\begin{definition}
Let $\ell\in \R$, $\rho \in \R_+^*$ and $(\lambda_n)_{n\in \N}$ be a strictly increasing sequence of positive numbers that diverges to infinity. We denote by $H(\rho,\ell,(\lambda_n)_n)$ the following statement:
\begin{align}
    H(\rho,\ell,(\lambda_n)) :\quad &\exists c,C>0,\ \forall n\in \N,\ \forall x\in \R^2,\nonumber\\ & \Proba{g+\ell\in \cross^h(x+[0,\rho\lambda_n]\times [0,\lambda_n]}\geq 1-Ce^{-c\lambda_n}.
\end{align}
Similarly, the statement $V(\rho,\ell, (\lambda_n)_n)$ is defined by
\begin{align}
    V(\rho,\ell,(\lambda_n)) :\quad & \exists c,C>0,\ \forall n\in \N,\ \forall x\in \R^2,\nonumber\\ &  \Proba{g+\ell\in \cross^v(x+[0,\lambda_n]\times [0,\rho\lambda_n]}\geq 1-Ce^{-c\lambda_n}.
\end{align}
Moreover, we denote by $H(\rho,\ell)$ the statement: 
\begin{align}
    H(\rho,\ell) :\quad & \exists C,c>0,\ \forall \lambda\geq 1,\ \forall x\in \R^2,\\& \Proba{g+\ell\in \cross^h(x+[0,\rho\lambda]\times [0,\lambda]}\geq 1-Ce^{-c\lambda}.
\end{align}
We similarly define $V(\rho,\ell)$ by
\begin{align}
    V(\rho,\ell) :\quad & \exists C,c>0,\ \forall \lambda\geq 1,\ \forall x\in \R^2,\\& \Proba{g+\ell\in \cross^v(x+[0,\lambda]\times [0,\rho\lambda]}\geq 1-Ce^{-c\lambda}.
\end{align}
\end{definition}
More informally, $H(\rho,\ell,(\lambda_n))$ says that rectangles of shape $\rho\lambda_n \times \lambda_n$ are crossed horizontally by $g+\ell\geq 0$ with very high probability. The event $H(\rho,\ell)$ may be understood by saying that one may choose the constants $c,C$ independent from the choice of a sequence $(\lambda_n)_n.$
We have the following results that come from gluing constructions.
\begin{lemma}
\label{lemma:grow_rectangle}
Let $\ell\in \R$, $\rho>\rho'>0$, $\varepsilon>0$ and $(\lambda_n)_n$ be a sequence of positive numbers that is strictly increasing and that diverges to infinity. In the following, we let $H(\rho)$ and $V(\rho)$  denote one of $H(\rho,\ell,(\lambda_n)_n)$ and $V(\rho,\ell, (\lambda_n)_n)$ or $H(\rho,\ell)$ and $V(\rho,\ell)$.
\begin{itemize}
    \item If $H(\rho)$ holds then so does $H(\rho')$. If $V(\rho)$ holds then so does $V(\rho').$
    \item If $H(1+\varepsilon)$ and $V(1)$ hold then so does $H(1+2\varepsilon)$. If $H(1)$ and $V(1+\varepsilon)$ holds then so does $V(1+2\varepsilon).$
\end{itemize}
\end{lemma}
\begin{figure}
    \centering
    \includegraphics[width=0.8\textwidth]{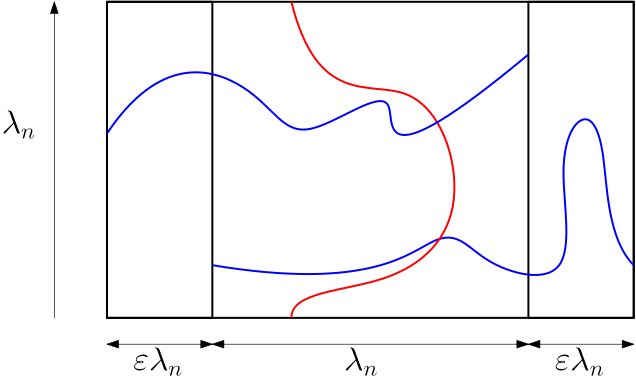}
    \caption{Illustration of the gluing construction in Lemma \ref{lemma:grow_rectangle}, in blue the horizontal crossings of $\mathcal{R}_1$ and $\mathcal{R}_2$ and in red the vertical crossing of $\mathcal{R}_3$.}
    \label{fig:grow}
\end{figure}
\begin{proof}
The proof of the first statement is immediate by inclusion of the corresponding rectangles. For the second statement we prove that  $H(1+\varepsilon,\ell, (\lambda_n)_n)$ and $V(1,\ell, (\lambda_n)_n)$ imply $H(1+2\varepsilon,\ell, (\lambda_n)_n)$ the proof of the other cases are completely similar. Let $x\in \R^2$ and $n\in \N$. We consider the the rectangles
$$\mathcal{R}_1 = x+[0,(1+\varepsilon)\lambda_n]\times [0,\lambda_n].$$
$$\mathcal{R}_2 = x+[\varepsilon \lambda_n,(1+2\varepsilon)\lambda_n]\times [0,\lambda_n].$$
$$\mathcal{R}_3 = x+[\varepsilon \lambda_n,(1+\varepsilon)\lambda_n]\times [0,\lambda_n].$$
By $H(1+\varepsilon,\ell, (\lambda_n)_n)$ and $V(1,\ell, (\lambda_n)_n)$ we may find constants $C,c>0$ (independent from $n$ and $x$) such that
\begin{align}
    \Proba{g+\ell \in \cross^h(\mathcal{R}_1)}\geq 1-Ce^{-c\lambda_n}, \\
    \Proba{g+\ell \in \cross^h(\mathcal{R}_2)}\geq 1-Ce^{-c\lambda_n}, \\
    \Proba{g+\ell \in \cross^v(\mathcal{R}_3)}\geq 1-Ce^{-c\lambda_n}.
\end{align}
By a union bound we find
\begin{equation}
    \Proba{g+\ell \in \cross^h(\mathcal{R}_1)\cap \cross^h(\mathcal{R}_2) \cap \cross^v(\mathcal{R}_3)}\geq 1-3Ce^{-c\lambda_n}.
\end{equation}
However, on this intersection event, then $g+\ell$ also belongs to $\cross^h(x+[0,(1+2\varepsilon) \lambda_n]\times [0,\lambda_n]$ (see Figure \ref{fig:grow}).
This concludes the proof since the constants $C,c>0$ do not depend on $n\in \N$ and $x\in \R^2$.

\end{proof}

\begin{lemma}
\label{lemma:smooth_rectangle}
Let $\ell\in \R, \rho>1$ and $(\lambda_n)_n$ be a sequence of positive numbers that is strictly increasing and that diverges to infinity. Assume furthermore that
\begin{equation}
    \sup_{n\in \N}\frac{\lambda_{n+1}}{\lambda_n}<\infty.
\end{equation}
Then,
\begin{align}
    H(\rho,\ell,(\lambda_n)_n) \text{ and } V(1,\ell, (\lambda_n)_n) \Rightarrow H(\rho,\ell), \\
    V(\rho,\ell,(\lambda_n)_n) \text{ and } H(1,\ell, (\lambda_n)_n) \Rightarrow V(\rho,\ell).
\end{align}
\end{lemma}
\begin{proof}
Let $C_0:= \sup_{n\in \N}\frac{\lambda_{n+1}}{\lambda_n} >0.$
We write $\rho=1+\varepsilon$ where $\varepsilon>0.$ By applying repeatedly Lemma \ref{lemma:grow_rectangle}, we see that $H(\rho C_0,\ell, (\lambda_n)_n)$ is verified. Let $\lambda\geq \lambda_0$. Let $n\in \N$ be such that $\lambda_n\leq \lambda< \lambda_{n+1}.$
Then we have $\rho\lambda \leq \rho\lambda_{n+1}\leq \rho C_0\lambda_n$ and $\lambda\geq \lambda_n$. This show that for $x\in \R^2$, $$g+\ell \in \cross^h(x+[0,\rho C_0\lambda_n]\times [0,\lambda_n]) \Rightarrow g+\ell \in \cross^h(x+[0,\rho\lambda]\times [0,\lambda]).$$
This concludes the proof of the lemma since
\begin{align*}
    &\Proba{g+\ell \in \cross^h(x+[0,\rho\lambda]\times [0,\lambda])} \\
    \geq & \Proba{g+\ell \in \cross^h(x+[0,\rho C_0\lambda_n]\times [0,\lambda_n])} \\
    \geq & 1-Ce^{-c\lambda_n} \\
    \geq & 1-Ce^{-c\lambda/C_0}
\end{align*}
where $c,C$ are constants that do not depend on $n$ (and therefore on $\lambda$) or $x$.
\end{proof}
Another consequence of gluing constructions is as follows.
\begin{lemma}
\label{lemma:cross_imply_giant}
Let $\ell\in \R$. If $H(\rho,\ell)$ and $V(\rho,\ell)$ hold for some $\rho>1$ then almost surely the set $\mathcal{E}_\ell(g)=\{g+\ell\geq 0\}$ contains a unique unbounded connected component and its complementary $\{g+\ell<0\}$ does not contain any unbounded connected component.
\end{lemma}
\begin{proof}
We provide two arguments. One proves the uniqueness of the eventual unbounded connected component of $\{g+\ell\geq 0\}$ and also proves the non existence of an unbounded connected component in $\{g+\ell<0\}.$ The second argument proves the almost sure existence of an unbounded connected component in $\{g+\ell\geq 0\}.$
For the first argument, for $n\in \N^*$ we consider the following rectangles
\begin{align*}
    \mathcal{R}_n^v := [-2^{n+1},-2^n]\times [-2^{n+1},2^{n+1}] \\
    \mathcal{R}_n^{v'} := [2^{n},2^{n+1}]\times [-2^{n+1},2^{n+1}] \\
    \mathcal{R}_n^h := [-2^{n+1},2^{n+1}]\times [-2^{n+1},-2^{n}] \\
    \mathcal{R}_n^{h'} := [-2^{n+1},2^{n+1}]\times [2^{n},2^{n+1}] \\
\end{align*}
Denote by $\mathcal{C}_n$ the following event:
\begin{equation}
    \mathcal{C}_n := \{g+\ell'\in \cross_v(\mathcal{R}_n^{v'})\cap \cross_v(\mathcal{R}_n^v)\cap\cross_h(\mathcal{R}_n^{h'})\cap\cross_h(\mathcal{R}_n^h) \}.
\end{equation}
By Lemma \ref{lemma:grow_rectangle} it appears that $H(4,q)$ and $V(4,q)$ hold. Therefore we may find constants $C,c>0$ (independent from $n$) such that by an union bound
$$\Proba{\mathcal{C}_n}\geq 1-Ce^{-c2^n}.$$
Therefore, by Borel-Cantelli Lemma, almost surely, there are infinitely many $n$ such that $\mathcal{C}_n$ occur. On this event, there cannot exists an unbounded connected component in $\{g+\ell<0\}$, therefore there exists at most one unbounded connected component in $\{g+\ell\geq 0\}.$
We now prove the second argument of this unbounded connected component.
For $n\in \N^*$ we denote by $\mathcal{D}_n$ the following event
\begin{equation}
    \mathcal{D}_n := \{g+\ell\in \cross^h([0,2^{2n+1}]\times [0,2^{2n}])\cap \cross^v([0,2^{2n+1}]\times [0,{2^{2n+2}}])\}.
\end{equation}
Observe that by an union bound we have
$$\Proba{\mathcal{D}_n}\geq 1-Ce^{-c2^{2n}}.$$
By Borel-Cantelli Lemma, almost surely, there exists a random $n_0\geq 1$ such that for $n\geq n_0$ then the event $\mathcal{D}_n$ occurs. This implies the existence of an unbounded connected component in $\mathcal{E}_\ell(g)=\{g+\ell\geq 0\}$ (see Figure \ref{fig:giant}).
\end{proof}
\begin{figure}[h]
    \centering
    \includegraphics[width=0.8\textwidth]{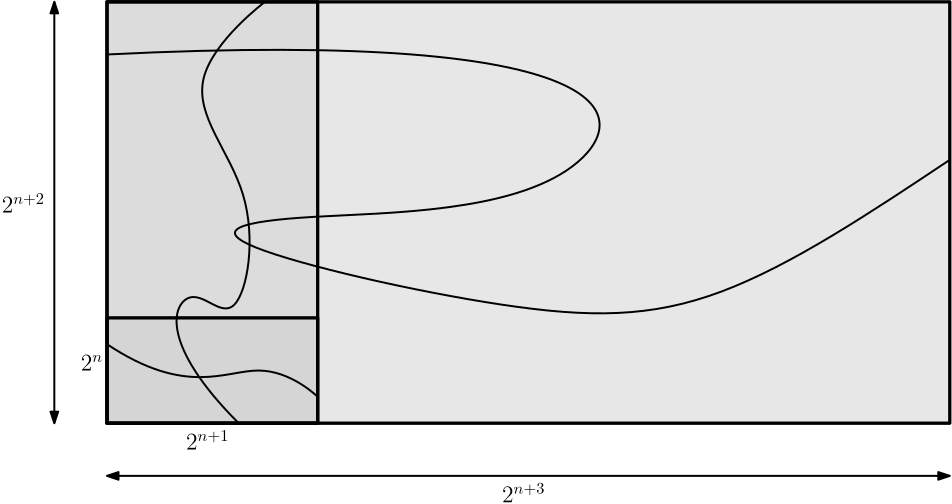}
    \caption{Illustration of the gluing construction in Lemma \ref{lemma:cross_imply_giant}.}
    \label{fig:giant}
\end{figure}

Thanks to this gluing theory, it appears that Theorem \ref{thm:crossings} is a consequence of the following proposition.
\begin{proposition}
\label{prop:crossings}
Let $g : \R^2 \to \R$ be a random function that admits $(g_R)_{R\in \N^*}$ a $(a,b)$-sequence of finite range approximations with $a>0$ and $b>1$ (see Definition \ref{def:approximations}). For $\rho>1$ we denote by $\mathcal{R}^h_\rho := [0,\rho]\times [0,1]$ and $\mathcal{R}^v_\rho = [0,\rho]\times [0,q]$.
Let $\ell\in \R$. Assume that there exists $\rho_0>1$ such that the following holds
\begin{align}
    \label{eq:hypothese_crossing}
    \inf_{x\in \R^2}\Proba{g+\ell\in \cross^h(x+\lambda\mathcal{R}^h_{\rho_0})}\xrightarrow[\lambda\to \infty]{}1. \\
    \label{eq:hypothese_crossing2}
    \inf_{x\in \R^2}\Proba{g+\ell\in \cross^v(x+\lambda\mathcal{R}^v_{\rho_0})}\xrightarrow[\lambda\to \infty]{}1.
\end{align}
Then, for any $\ell'>\ell$ and for any $\rho>0$ the two propositions $H(\rho,\ell')$ and $V(\rho,\ell')$ are verified.
\end{proposition}
Proposition \ref{prop:crossings} is sometimes referred to as a \textit{bootstrap} result. Indeed, we assume that the probability of crossing some rectangles goes to $1$ and, by paying a little sprinkling from $\ell$ to $\ell'$, we quantify the speed of this convergence by showing that it is exponentially fast.
We show how to deduce Theorem \ref{thm:crossings} from Proposition \ref{prop:crossings}.
\begin{proof}[Proof of Theorem \ref{thm:crossings}]
By Proposition \ref{prop:crossings} we see that for $\rho>1$ and $\ell'<\ell$ then both $H(\rho,\ell')$ and $V(\rho,\ell')$ hold. Therefore the conclusion of Theorem \ref{thm:crossings} directly follows from Lemma \ref{lemma:cross_imply_giant}.
\end{proof}

In order to prove Proposition \ref{prop:crossings} we introduce the notion of increasing and decreasing events.
\begin{definition}
Let $S\subset \R^2$ be a subset. Denote by $\mathcal{F}(S,\R)$ the set of functions from $S$ to $\R$. A subset $A\subset \mathcal{F}(S,\R)$ is said to be \textit{increasing} if
$$\forall g\in A,\ \forall h\in \mathcal{F},\ g\leq h\Rightarrow h\in A,$$
where $g\leq h$ means $g(x)\leq h(x)$ for all $x\in S$.
A subset $A\subset \mathcal{F}(S,\R)$ is said to be an increasing and supported on $S'\subset S$ if there exists $A'\subset \mathcal{F}(S',\R)$ increasing such that
$$A=\{g : S\to \R\ |\ g_{|S'}\in A'\},$$
where $g_{|S'}$ denotes the restriction of the function $g$ to $S'$.
The notions of decreasing events, and decreasing events supported on $S'$ are defined similarly.
\end{definition}
We will use the following technical lemma.
\begin{lemma}
\label{lemma:decor}
Let $g : \R^2 \to \R$ be a random function that admits $(g_R)_{R\in \N^*}$ a $(a,b)$-sequence of finite range approximations with $a>0$ and $b>0$ (see Definition \ref{def:approximations}). Let $S_1,S_2\subset \R^2$ such that $R := d(S_1,S_2)\geq 1$. Let $A$ be a decreasing event supported on $S_1$ and $B$ be a decreasing event supported on $S_2$. Then, for any $\ell\in \R$ and any $\varepsilon>0$
\begin{align}
    \Proba{g+\ell \in A\cap B}\leq &\Proba{g+\ell-2\varepsilon \in A}\Proba{g+\ell-2\varepsilon\in B}\nonumber\\
    &+4CNe^{-c\varepsilon^a R^b},
\end{align}
where $N$ is the minimal number of boxes of the form $x+[0,1]^2$ needed to cover $S_1\cup S_2$ and $C,c>0$ are constants that only depend on $g$.
\end{lemma}
\begin{proof}
We observe that there are two cases.
Either $\sup_{S_1\cup S_2}|g-g_R|>\varepsilon$, this is the "bad" case and the probability of such event is upper bounded by $CNe^{-c\varepsilon^a R^b}$ by Definition \ref{def:approximations}. Otherwise, $\sup_{S_1\cup S_2}|g-g_R|\leq \varepsilon$. Since $A\cap B$ is a decreasing event supported on $S_1\cup S_2$, and since we have $g_R\leq g+\varepsilon$ on $S_1\cup S_2$ we have $g+\ell \in A\cap B\Rightarrow g_R+\ell-\varepsilon\in A\cap B.$
We can then use the independence of the events $g_R+\ell-\varepsilon\in A$ and $g_R+\ell-\varepsilon\in B$ which is guaranteed by Definition \ref{def:approximations}. More formally we have

\begin{align*}
    \Proba{g+\ell \in A\cap B} \leq & \Proba{g+\ell\in A\cap B, \sup_{S_1\cup S_2}|g-g_R|>\varepsilon} \\
    & + \Proba{g+\ell \in A\cap B, \sup_{S_1\cup S_2}|g-g_R|\leq \varepsilon} \\
    \leq & \Proba{\sup_{S_1\cup S_2}|g-g_R|>\varepsilon} \\
    &  + \Proba{g_R+\ell-\varepsilon \in A\cap B} \\
    \leq & \Proba{\sup_{S_1\cup S_2}|g-g_R|>\varepsilon} \\
    &  + \Proba{g_R+\ell-\varepsilon \in A}\Proba{g+R+\ell-\varepsilon\in B} \\
\end{align*}
By the same argument we get
\begin{align*}
    \Proba{g_R+\ell-\varepsilon \in A} \leq & \Proba{\sup_{S_1\cup S_2}|g-g_R|>\varepsilon} \\
    &  + \Proba{g+\ell-2\varepsilon \in A} \\
    \Proba{g_R+\ell-\varepsilon \in B} \leq & \Proba{\sup_{S_1\cup S_2}|g-g_R|>\varepsilon} \\
    &  + \Proba{g+\ell-2\varepsilon \in B} \\
\end{align*}
Putting these inequalities together yields the conclusion.
\end{proof}
We now provide the proof of Proposition \ref{prop:crossings} which is inspired by the ideas from \cite{MV20}.
\begin{proof}[Proof of Proposition \ref{prop:crossings}]
Recall that we denote by $a>0$ and $b>1$ are the two exponents from Definition \ref{def:approximations}.
Since $b>1$ we may choose $\delta\in ]0,1[$ such that
\begin{equation}
    b\delta>1.
\end{equation}
We also may take $\gamma>0$ small enough such that
\begin{equation}
    \label{eq:betadelta_alphagamma}
    b\delta-a\gamma>1.
\end{equation}
Take $\lambda_0>1$ (to be fixed later). For $n\in \N^*$ we define $\lambda_n$ by induction,
\begin{equation}
    \label{eq:rec_lambda}
    \forall n\in \N^*,\ \lambda_n := 2\lambda_{n-1}+\lambda_{n-1}^\delta.
\end{equation}
We claim that there exists a constant $C_{\delta, \lambda_0}>0$ such that
\begin{equation}
    \label{eq:order_lambda}
    \lambda_n \sim_{n\to \infty}C_{\delta,\lambda_0}2^n.
\end{equation}
In fact, let $c>2$ be a constant so that $c^\delta <2$ (this is possible since $\delta<1$). Let $n_0$ (depending on $\delta,c,\lambda_0$) be such that
$$\forall n\geq n_0,\ \lambda_{n+1}\leq c\lambda_n.$$
(this is possible because $\delta<1,$ $c>2$ and since $\lambda_n \to \infty$). This yields that $\lambda_n \leq \tilde{C}_{\delta,\lambda_0}c^n$ for all $n\in \N^*$ where $\tilde{C}_{\delta,\lambda_0}$ is a constant. Dividing \eqref{eq:rec_lambda} by $2^{n}$ we obtain that for $n\in \N$
$$\left|\frac{\lambda_{n+1}}{2^{n+1}} - \frac{\lambda_n}{2^n} \right|=\frac{\lambda_n^\delta}{2^{n+1}}\leq\tilde{C}^\delta_{\delta,\lambda_0} \left(\frac{c^\delta}{2}\right)^n$$
By cancellation (and since $\frac{c^\delta}{2}<1$) we obtain \eqref{eq:order_lambda}.
Let $\ell,\ell'\in \R$ such that $\ell<\ell'$. We introduce a sequence $(\ell_n)_n$ defined by
We define
\begin{equation}
    \ell_n := \ell'-\frac{\ell'-\ell}{\lambda_n^\gamma}.
\end{equation}
We observe that for any $n\in \N$,
\begin{equation}
    \ell < \ell_n < \ell_{n+1}< \ell'.
\end{equation}
We introduce for $n\in \N$ the quantity $\varepsilon_n$ defined by,
\begin{equation}
    \forall n\in \N,\ \varepsilon_n := \ell_{n+1}-\ell_n,
\end{equation}
and we observe that
\begin{equation}
    \forall n\in \N,\ \varepsilon_n \geq (\ell'-\ell)(1-2^{-\gamma})\lambda_n^{-\gamma}.
\end{equation}
For $n\in \N$ we write,
\begin{equation}
    p_n^h := \sup_{x\in \R^2} \Proba{g+\ell_n \not \in \cross^h(x+[0,2\lambda_n]\times [0,\lambda_n])},
\end{equation}
\begin{equation}
    p_n^v := \sup_{x\in \R^2} \Proba{g+\ell_n \not\in \cross^v(x+[0,\lambda_n]\times [0,2\lambda_n])}.
\end{equation}
We also denote by $p_n$ the following quantity,
\begin{equation}
    \forall n\in \N,\ p_n := \max(p_n^h,p_n^v).
\end{equation}
Consider the gluing construction shown in Figure \ref{fig:multicross}.
\begin{figure}[h]
    \centering
    \includegraphics[width=0.8\textwidth]{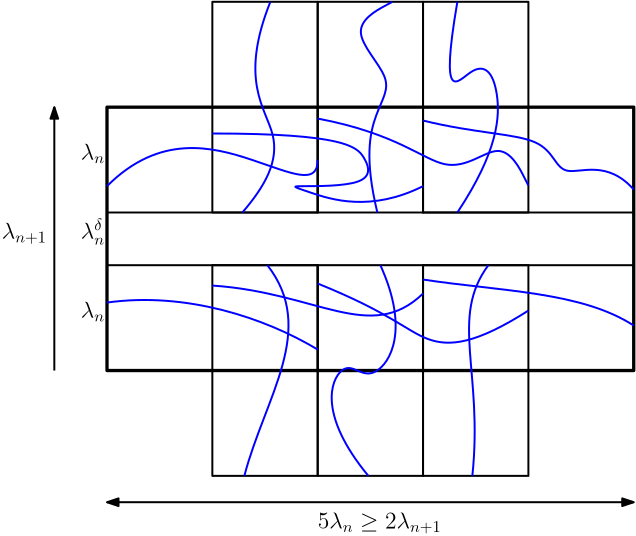}
    \caption{Illustration of the gluing construction in the proof of Proposition \ref{prop:crossings}.}
    \label{fig:multicross}.
\end{figure}
Since $5\lambda_n \geq 2\lambda_{n+1}$, it appears that the crossing of a rectangle of the form $[2\lambda_{n+1}]\times [0,\lambda_n]$ may be induced in two ways, one way is the occurrence of seven horizontal and vertical crossings of rectangles of shape $2\lambda_n \times \lambda_n$ or $\lambda_n\times 2\lambda_n$ located in the upper part of the rectangle $[2\lambda_{n+1}]\times [0,\lambda_n]$. The other way is also the occurrence of seven such crossings but located in the bottom part of the rectangle $[2\lambda_{n+1}, \lambda_n]$. Between the rectangles used by one way and the rectangles used by the other way, there is a distance of $\lambda_n^\delta.$
Therefore, one may apply Lemma \ref{lemma:decor}, together with \eqref{eq:betadelta_alphagamma} to see that for all $n\in \N$.
\begin{equation}
    p^h_{n+1}\leq 49p_n^2+C'_1\lambda_{n}^2e^{-c_1\lambda_n^{-a\gamma}\lambda_n^{\delta b}} \leq 49p_n^2 + C_1e^{-c_1\lambda_n}
\end{equation}
where $C_1,C'_1,c_1>0$ are constants that depend on $\lambda_0, a,b,\gamma,\delta,\ell,\ell'$. We may also assume without loss of generality that $C_1>1$.
We may do a similar construction of vertical crossings and altogether we get
\begin{equation}
    \forall n\in \N,\ p_{n+1}\leq 49p_n^2 + C_1e^{-c_1\lambda_n}.
\end{equation}
We introduce a sequence $(u_n)_n$ defined by,
\begin{equation}
    \forall n\in \N, \ u_n:= 49(p_n +C_1e^{-c_1\lambda_n/2}).
\end{equation}
Then,
\begin{align*}
    u_{n+1} &= 49(p_{n+1} +C_1e^{-c_1\lambda_{n+1}/2}) \\
    &\leq 49(49p_n^2+C_1e^{-c_1\lambda_n}+C_1e^{-c_1\lambda_{n+1}/2}) \\
    &\leq 49^2p_n^2+98C_1e^{-c_1\lambda_n} \\
    &\leq (49p_n)^2+(49C_1e^{-c_1\lambda_n/2})^2 \\
    &\leq u_n^2.
\end{align*}
Therefore we have
$$u_n \leq (u_0)^{2^n}.$$
Since $\ell_0\geq \ell$, then by \eqref{eq:hypothese_crossing} and \eqref{eq:hypothese_crossing2} we see that we may fix $\lambda_0$ big enough so that $u_0<1.$
This shows that there exists a constant $c>0$ such that
\begin{equation}
    \forall n\in \N,\ u_n \leq e^{-c2^n}.
\end{equation}
Consequently, we may find constants $C,c>0$ such that for any $n\in \N$,
\begin{align*}
    p_n \leq Ce^{-c2^n}+Ce^{-c\lambda_n}.
\end{align*}
By \eqref{eq:order_lambda} this implies
\begin{equation}
    p_n \leq Ce^{-c\lambda_n.}
\end{equation}
Since for all $n\in \N$ we have $\ell_n\leq \ell'$ then we have proven that $H(2,\ell',(\lambda_n))$ and $V(2,\ell',(\lambda_n))$ hold. By Lemma \ref{lemma:grow_rectangle} and \ref{lemma:smooth_rectangle} this is enough to get the conclusion.
\end{proof}

\bibliography{biblio.bib}

David, Vernotte
Institut Fourier, UMR 5582, Laboratoire de MathématiquesUniversité Grenoble Alpes, CS 40700, 38058 Grenoble cedex 9, France

\end{document}